\documentclass[11pt]{article}

\usepackage{latexsym,mathrsfs}
\usepackage{amsmath,amssymb} 
\usepackage{amsthm,enumerate,verbatim}
\usepackage{amsfonts}
\usepackage{graphicx}
\usepackage{algorithm}
\usepackage{algorithmic}

\usepackage{url}
\usepackage[dvipsnames]{xcolor}
\usepackage{pifont}
\usepackage{mathtools}
\usepackage{makecell}
\usepackage[normalem]{ulem}
\usepackage{hyperref}

\newtheorem{conj}{Conjecture}
\newtheorem{lemma}{Lemma}

\newtheorem{theorem}{Theorem}
\newtheorem{corollary}{Corollary}
\newtheorem{definition}{Definition}
\newtheorem*{definition*}{Definition}
\newtheorem{example}{Example}

\newtheorem{assumption}{Assumption}

\DeclareMathOperator{\conv}{conv}

\DeclareMathOperator{\cone}{cone} 
\DeclareMathOperator{\rank}{rank}

\DeclareMathOperator{\logdet}{logdet} 
\DeclareMathOperator{\tr}{tr} 
\DeclareMathOperator{\diag}{Diag} 
\DeclareMathOperator{\vol}{vol}

\definecolor{brightpink}{rgb}{1.0, 0.0, 0.5}

\definecolor{orange}{rgb}{1, 0.5, 0}

\title{Dual Simplex Volume Maximization for \\ 
Simplex-Structured Matrix Factorization} 
\date{}

\author{
Maryam Abdolali\thanks{Email: maryam.abdolali@kntu.ac.ir} \\ K.N.Toosi University (KNTU) \\ 
Tehran, Iran 
 \and 
Giovanni Barbarino
\thanks{Email: giovanni.barbarino@umons.ac.be. GB acknowledges the support by the European Union (ERC consolidator, eLinoR, no 101085607). GB is member of the Research Group GNCS (Gruppo Nazionale per il Calcolo Scientifico) of INdAM (Istituto Nazionale di Alta Matematica).}
\qquad 
Nicolas Gillis\thanks{Email: nicolas.gillis@umons.ac.be. NG acknowledges the support by the European Union (ERC consolidator, eLinoR, no 101085607).} \\ University of Mons \\ Mons, Belgium
	}

\begin{document}

\maketitle

\begin{abstract}
	Simplex-structured matrix factorization (SSMF) is a generalization of nonnegative matrix factorization, a fundamental interpretable data analysis model, and has applications in hyperspectral unmixing and topic modeling. 
 To obtain identifiable solutions, a standard approach is to find minimum-volume solutions. 
 By taking advantage of the duality/polarity concept for polytopes, we convert minimum-volume SSMF in the primal space to a maximum-volume problem in the dual space. 
 We first prove the identifiability of this maximum-volume dual problem. Then, we use this dual formulation to provide a novel optimization approach which bridges the gap between two existing families of algorithms for SSMF, namely volume minimization and facet identification. 
 Numerical experiments show that the proposed approach performs favorably compared to the state-of-the-art SSMF algorithms. 
\end{abstract}


\section{Introduction}
Matrix factorization (MF) is a fundamental technique for extracting latent low-dimensional factors, with applications in numerous fields, such as data analysis, machine learning and signal processing. 
MF aims to decompose a given data matrix, $X \in \mathbb{R}^{m\times n}$, where the $n$ columns represent $m$-dimensional samples, 
into the product of two smaller matrices, $W \in \mathbb{R}^{m \times r}$ and $H \in \mathbb{R}^{r \times n}$ called factors, such that $X \approx WH$. Often imposing additional constraints, such as sparsity or nonnegativity, on the factors is crucial, e.g., for interpretation purposes, leading to structured (or constrained) matrix factorization (SMF); see, e.g., \cite{udell2016generalized, fu2020computing} and the references therein.  
A specific problem of the broad family of SMF assumes that each column of $H$ belongs to the unit simplex, that is, for all $j$, 
	\begin{equation*}
		H(:,j) \; \in \; \Delta^r := \Big\{x \in \mathbb{R}^{r} \ \big| \ x \geq 0, e^\top x = \sum_{i=1}^r x_i = 1 \Big\}, 
	\end{equation*} 
 where $e$ is the vector of all ones of appropriate dimension. 
SSMF has several applications in machine learning with two prominent examples including unmixing hyperspectral images where $H(i,j)$ is the proportion/abundance of the $i$th material within the $j$th pixel~\cite{heinz2001fully, bioucas2012hyperspectral, ma2013signal}, 
and topic modeling where where $H(i,j)$ is the contribution of the $i$th topic within the $j$th document~\cite{arora2013practical, fu2018anchor, bakshi2021learning}.  

\paragraph{Contribution and outline of the paper}

 This paper focuses on the concept of duality and uses the correspondence between primal and dual spaces to provide a new perspective on fitting a simplex to the samples. The main contributions are as follows:
 \begin{itemize}
 	\item We present a new formulation for SSMF which is based on the concept of duality. This formulation provides a different perspective on SSMF and bridges the gap between two existing families of approaches: volume minimization and facet-based identification (Section~\ref{sec:proposedmodel}). 
  
  \item We study the identifiability of the parameters with this new formulation (Section~\ref{sec:identif}). 
  
  \item We develop an efficient optimization scheme based on block coordinate descent (Section~\ref{sec:optim}). 
  
 	\item We provide numerous numerical experiments on both synthetic and real-world data sets, showing that the proposed algorithm competes favorably with the state of the art (Section~\ref{sec:exp}). 
 \end{itemize}

\section{Previous works}

In this paper, we consider the following SSMF formulation: Given $X \in \mathbb{R}^{m \times n}$ and $r>0$, solve 
\[
\min_{W \in \mathbb{R}^{m \times r}, H \in \mathbb{R}^{r \times n}} 
\| X -  WH \|_F^2 
\quad \text{ such that } \quad 
H(:,j) \in \Delta^r \text{ for all } j. 
\] 
SSMF is closely related to NMF which decomposes a nonnegative matrix, $X \geq 0$, as $X = WH$ where $W \geq 0$ and $H \geq 0$~\cite{lee1999learning, gillis2020nonnegative}.  
In fact, normalizing each column of $X$ to have unit $\ell_1$ norm, and assuming w.l.o.g.\ that the columns of $W$ also have unit $\ell_1$ norm, implies that the columns of $H \geq 0$ also have $\ell_1$ norm, since 
$e^\top = 
e^\top X = 
e^\top WH 
= e^\top H$, and hence $H$ is column stochastic. 

\subsection{Geometric interpretation of SSMF and uniqueness/identifiability} \label{sec:geominterpret}

For an exact SSMF decomposition, we have 
\[
X(:,j) = W H(:,j) = \sum_{k=1}^r W(:,k) H(k,j), 
\]
meaning that the columns of $X$ are convex combinations of the column of $W$. In other words, SSMF aims to find $r$ vectors, $\{W(:,k)\}_{k=1}^r$, such that their convex hull contains the columns of $X$, that is, for all $j$ 
\[
X(:,j) \quad \in \quad \conv(W) = \{ x \ | \ x = Wh, h \in \Delta^r \}. 
\]
We will say that $X=WH$ has a unique SSMF if any other SSMF of $X$, say $X = W'H'$, can only be obtained by permutation of the columns of $W$ and rows of $H$, that is, $X = W'H'$ implies that $W'(:,k) = W(:,\pi_k)$ and $H'(k,:) = H(\pi_k,:)$ for some permutation $\pi$ of $\{1,2,\dots,r\}$. 
Without any further constraints, SSMF is never unique, because we can always enlarge the convex hull of $W$ to contain more points, and hence obtain equivalent factorizations~\cite{gillis2015exact}. 
It is therefore crucial for SSMF models to include additional constraints or regularizers to obtain identifiable models. 
There has been three main approaches to achieve this goal: separability, volume minimization and facet-based identification. 
They are described in the next three sections.

\subsection{Separability}  \label{sec:intro:separability}

Separability assumes that the columns of $W$ are among the columns of $X$~\cite{arora2012computing,bioucas2012hyperspectral}, that is, 
 the matrix $X$ admits an SSMF of the form $X = WH$ with $W = X(:,\mathcal{K})$ and the index set $\mathcal{K}$ contains $r$ elements, that is, $|\mathcal{K}| = r$. Equivalently,  $X = WH$ with $H(:,\mathcal{K}) = I_r$ for some  index set $\mathcal{K}$, where $I_r$  the identity matrix of dimension $r$. See Figure~\ref{fig:GEOsepminvolsparse} (left) for an illustration. 
\begin{figure}[ht!]
	\begin{center}
		\includegraphics[width=\textwidth]{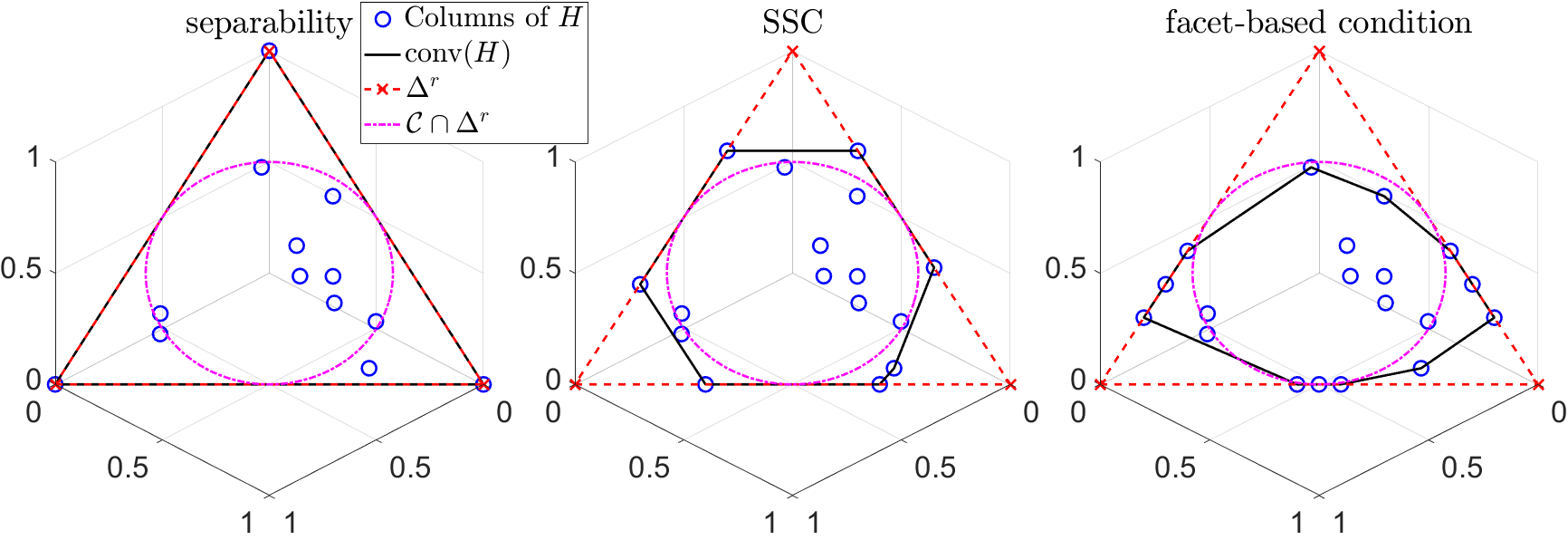}   
		\caption{Comparison of separability (left), SSC (middle), and the facet-based condition (right) for the matrix $H$ whose columns lie on $\Delta^r$ in the case $r = 3$. 
			On the left, separability requires the columns of $H$ to contain the unit vectors, that is, $H(:,\mathcal{K}) = I_r$ for some $\mathcal{K}$.
			On the middle, the SSC requires $\mathcal{C} \subset \cone(H)$. 
			On the right, the facet-based condition requires $r=3$ columns of $H$ on each facet of the unit simplex. Figure from~\cite{abdolali2021simplex}. \label{fig:GEOsepminvolsparse}
   }  
	\end{center}
\end{figure} 
In other words, separability requires that for each basis vector, $W(:,k)$, there exists a data point, $X(:,\mathcal{K}_k)$, such that $W(:,k) = X(:,\mathcal{K}_k)$. This is the so-called pure-pixel assumption in hyperspectral unmixing~\cite{boardman1995mapping}, and the anchor-word assumption in topic modeling~\cite{arora2013practical}. 

Separability leads to identifiability, and 
simplifies the problem resulting in polynomial-time algorithms, some running in $O(mnr)$ operations, with theoretical guarantees; see~\cite[Chapter 7]{gillis2020nonnegative} for a comprehensive survey on these algorithms. 
However, separability is a strong assumption which might not hold in all real-world scenarios.

\subsection{Volume minimization} \label{sec:volmin} 

In order to relax separability, one can look for an SSMF, $X = WH$, where the volume of the convex hull of the columns of $W$ and the origin within the column space of $W$, which is proportional to $\det\big (W^\top W\big )$, is minimized. 
The first intuitions and empirical evidences came from the hyperspectral imaging literature~\cite{craig1994minimum, miao2007endmember}. 
Later, minimum-volume SSMF was shown to be identifiable~\cite{fu2015blind, lin2015identifiability} under the so-called sufficiently scattered condition\footnote{There exist  several definitions of the SSC, with minor variations. The main condition, $\mathcal{C} \subset \cone(H)$, is always required.} (SSC) introduced in \cite{huang2013non}: 
\begin{definition}[SSC] \label{def:SSC}
	The matrix $H \in \mathbb{R}^{r \times n}$ satisfies the SSC if its conic hull, defined by 
 $
 \text{cone}(H)
 = \big\{y \ | \ y = Hx, x \geq 0 \big\}
 $, contains the second-order cone 
 $
 \mathcal{C} = \big\{x \in \mathbb{R}_+^r  \ | \ e^\top x \geq \sqrt{r-1} ||x||_2 \big\}
 $. 
 Moreover, the only real orthogonal matrices $Q \in \mathbb{R}^{r \times r}$ satisfying $\cone(H) \subset \cone(Q)$ are permutation matrices. 
\end{definition}
Intuitively, this assumption implies that the columns of the matrix $H$ are well scattered in the unit simplex $\Delta^r$. For example, the SSC implies that there are at least $r-1$ columns of $H$ on each facet of $\Delta^r$, meaning that $H$ has at least $r-1$ zeros per row.   
See Figure~\ref{fig:GEOsepminvolsparse} for an illustration, and~\cite{fu2019nonnegative} and \cite[Chapter 4]{gillis2020nonnegative} for more details.

Many algorithms have been designed using volume minimization, starting from~\cite{craig1994minimum}. 
A common approach is to minimize the volume of enclosing vertices $W$ by minimizing the  determinant of $W^\top W$~\cite{fu2015blind, lin2015identifiability}: 
\begin{equation}\label{min_vol}
	\min_{W,H} \det(W^\top W) \qquad \text{such that} \qquad X=WH \ \text{and} \ H(:,j)\in \Delta^r \ \text{for all }j. 
\end{equation} 
Under the SSC, solving~\eqref{min_vol} guarantees to recover the columns of $W$ and $H$ in the SSMF $X = WH$, up to permutation~\cite{fu2015blind, lin2015identifiability}. 
In the presence of noise, one has to balance the data fitting term and the volume regularizer by minimizing 
$\|X - WH\|_F^2 + \lambda \det(W^\top W)$ for some well-chosen penalty parameter $\lambda > 0$.  
Another approach is minimum-volume
enclosing simplex (MVES)~\cite{chan2009convex} that attempts to simplify the problem  by focusing on the volume of a dimension-reduced transformation of $W$ (via the SVD), say $\bar W \in \mathbb{R}^{r \times r-1}$; see Section~\ref{sec:preproc} for details. MVES works with the transformed matrix 
\begin{equation}\label{eq:tW} 
\Tilde{W}  = 
[ \bar W(:,1)-\bar W(:,r), \dots , \bar W(:,r-1)-\bar W(:,r)]  
\in \mathbb{R}^{(r-1) \times (r-1)}, 
\end{equation}
and minimizes $|\det(\tilde W)| = \vol\big(\conv(\bar W) \big)$.  
This reformulation allows them to solve the subproblem in each column of $W$ via alternating linear optimization. 
More recently, a more general class of problems is considered in~\cite{tatli2021polytopic}, where the columns of $H$ are restricted to belong to a polytope, which is referred to as polytopic matrix factorization. 
Instead of minimizing the volume of $\conv(W)$, the determinant of $HH^\top$ is maximized, with identifiability guarantees under a generalized SSC. 

In contrast to separable-based algorithms,  volume-minimization problems, such as~\eqref{min_vol}, are not convex, and hence it is not straightforward to solving them up to global optimality. Hence although volume minimization allows one to theoretically identify SSMF under relaxed conditions, it makes the optimization problems harder to solve than under the separability assumption. Also, robustness to noise is not well understood.

\subsection{Facet-based identification} 

Instead of looking for columns of $W$ whose convex hull contains the columns of $X$, 
one can instead look for a set of facets (a facet is an affine hyperplane $\{ x \ | \ a^\top x = b\}$
delimiting the associated half space $\{ x \ | \ a^\top x \leq b\}$ for some vector $a$ and scalar $b$) enclosing a region where the columns of $X$ lie.   
Two main algorithms in this category are the following: 
\begin{itemize}	

\item{Minimum-volume inscribed ellipsoid (MVIE)}~\cite{lin2018maximum} identifies the enclosing $r$ facets by a two-step approach: (i)~generate all facets of $\conv(X)$, and (ii)~find the maximum-volume ellipsoid inscribed in the generated facets. 
Under the SSC, this ellipsoid touches every facet of $\conv(W)$ which leads to the identification of $r$ facets of the simplex and, subsequently, the vertices in $W$. Although MVIE is guaranteed to recover $W$ in the noiseless case under the SSC, it relies on the computationally expensive algorithm of facet enumeration limiting the algorithm values of $r$ up to around $10$, and is sensitive to noise and outliers. 

 \item{Greedy facet-based polytope identification (GFPI)}~\cite{abdolali2021simplex} uses duality to map the facet identification problem in the primal space into the corresponding vertex identification problem in the dual space. Using duality, GFPI prioritizes facets with most samples on them. GFPI formulates this problem as a mixed integer program which identifies the facets sequentially. 
 GFPI has several significant advantages over other approaches, including the ability to handle rank-deficient matrices, outliers, and input data that violates the SSC. Moreover, it is identifiable under a typically weaker condition than the SSC, namely the facet-based condition (FBC) which requires $r$ data points on each facet of $\conv(W)$ (and some other minor conditions generically satisfied); 
 see Figure~\ref{fig:GEOsepminvolsparse} for an illustration and~\cite{abdolali2021simplex} for more details. 
\end{itemize}

In the next section, we introduce our novel approach which 
relies on facet-based identification. 
It is based on a novel efficient vertex enumeration in the dual space. 
In contrast to GFPI, the proposed approach does not rely on greedy sequential identification of vertices (corresponding facets in the primal space) but identifies the facets simultaneously by maximizing their volume in the dual space. It is presented in the next section.

\section{Proposed model: SSMF based on maximum-volume in the polar} \label{sec:proposedmodel}

Our proposed approach is based on duality/polarity (we use both words interchangeably in this paper). 
In order to recover the columns of $W$, which are the vertices of the simplex enclosing the columns of $X$, we focus on extracting the facets of its convex hull. The facets are implicitly obtained by calculating the vertices of the corresponding dual simplex. Before explaining this in Section~\ref{sec:polar}, 
we first reduce the dimension to work with full-dimensional polytopes.  
This reduction requires $\conv(X)$ to have dimension $r-1$ which requires the dimension of its affine hull to be $r-1$, which we will assume throughout the paper. 
\begin{assumption} \label{ass:dimaffhull}
    The affine hull of $X$, $\{  y \ | \ y = Xh \text{ where } e^\top h = 1 \}$, has dimension $r-1$. 
\end{assumption} 
If $\rank(H) = r$, which is implied by the SSC condition, and if the affine hull of $W$ has dimension $r-1$, then the affine hull of $X$  has dimension $r-1$. 
Note that $\rank(W) = r$ implies that the affine hull of $W$ has dimension $r-1$. However, we could also have the case $\rank(W) = r-1$ if $0$ belongs to the affine hull of $W$ (e.g., in 2 dimensions, $\conv(W)$ is a triangle containing the origin).

\subsection{Preprocessing: translation and dimensionality reduction}  \label{sec:preproc}

In this paper, like in many other SSMF approaches, e.g., 
MVES~\cite{chan2009convex} 
and GFPI~\cite{abdolali2021simplex}, see also~\cite{ma2013signal}, we will use a preprocessing of the data to reduce it to an $(r-1)$-dimensional space. 
This has several motivations: 
\begin{itemize}
    \item The convex hull of the columns of $W$, $\conv(W)$, is an $r-1$ dimensional simplex, under Assumption~\ref{ass:dimaffhull}.  

    \item In the noiseless case, the preprocessing does not change the geometry and properties of the problem. In noisy settings, it filters noise via dimensionality reduction. 

    \item The notion of polarity is simpler to grasp for full-dimensional polytopes: the polar of $\conv(W)$ will also be an $(r-1)$-dimensional polytope. 
    
\end{itemize}

The preprocessing has two steps. 

\paragraph{Step 1: Translation around the origin.}

Let us choose a point, $v$, in the relative interior of $\conv(X)$. 
For example, one can choose the sample mean, $v = \bar{x} = \frac{1}{n} \sum_{j=1}^n X(:,j)$. We will discuss in Section~\ref{sec:identif} the importance of this choice, which will need to be part of the optimization problem to obtain identifiability.  
 The first step for preprocessing the data is to remove $v$ from each sample to obtain $\hat{X} = X - v e^\top$. 
Let $X = WH$ be an SSMF of $X$ where $e^\top H = e^\top$ and $H \geq 0$. 
 This first step simply amounts to translating the SSMF problem, since 
 \[
\hat{X} = X - v e^\top 
= WH - ve^\top 
= [W - ve^\top] H = \hat W H, \text{ with } \hat W =  W - ve^\top. 
 \]
 Since $v$ is in the relative interior of $\conv(X)$, the vector of zeros is in the relative interior of $\conv(\hat{X})$: $v = Xh$ for some $h > 0$ and $e^\top h = 1$  implying   
 \[
0 = Xh - v = [X-v e^\top] h = \hat X h. 
 \]
 This shows that the column space of $\hat X$ has dimension $r-1$,  
under Assumption~\ref{ass:dimaffhull}.  

 \paragraph{Step 2: Dimensionality reduction} The second step is to project the centered samples $\hat{X}$  onto the $(r-1)$-dimensional column space of $\hat{X}$ 
 using the truncated SVD. 
 Let $U \Sigma V^\top$ be the truncated SVD of $\hat{X}$ where $U \in \mathbb{R}^{m \times (r-1)}$, $\Sigma \in \mathbb{R}^{(r-1)\times (r-1)}$ and $V \in \mathbb{R}^{n \times (r-1)}$. The projected samples $Y \in \mathbb{R}^{(r-1) \times n}$ are obtained by: $Y = U^\top\hat{X} = \Sigma V^\top$. 
The second step of the preprocessing simply premultiplies $\hat{X}$ by $U^\top$, to obtain 
\[
Y = U^\top \hat{X} = (U^\top \hat{W}) H = P H, 
\quad \text{ with } P = U^\top \hat W =  U^\top [W - ve^\top] \in \mathbb{R}^{(r-1) \times r} .  
\] 
This is also an equivalent SSMF of smaller dimension, with the same matrix $H$. 
In the presence of noise, this preprocessing can help filter the noise. 
Note that in the presence of non-Gaussian noise, one might project using other norms, that is, not use the SVD which is based on the $\ell_2$ norm  but low-rank matrix approximations minimizing other norms, e.g.,~\cite{candes2011robust, gillis2018complexity}.

\subsection{Polar representation} \label{sec:polar}

We have now transformed the original rank-$r$ SSMF problem 
of matrix $X \in \mathbb{R}^{m \times n}$  
into an equivalent SSMF problem of a rank-$(r-1)$ matrix $Y \in \mathbb{R}^{r-1 \times n}$. 

Let us show how to construct a polar formulation of this problem. Any feasible solution $(P,H)$ of SSMF for $Y$ 
satisfies $Y = P H$ where $P \in \mathbb{R}^{r-1 \times r}$ and $H(:,j) \in \Delta^r$ for all $j$. By the geometric interpretation of SSMF, see Section~\ref{sec:geominterpret},   
$\conv(Y) \subseteq \conv(P)$. Let us define the polar of a set. 
\begin{definition}[Polar]
    Given any set $\mathcal S \subseteq \mathbb R^d$, 
    its \textit{polar}, denoted $\mathcal S^*$, is defined as 
    \[
    \mathcal S^* := \left\{ \theta \in \mathbb R^d \, \big| \, \theta^\top x \le 1\,\, \text{ for all } x\in \mathcal S  \right\}. 
    \]
\end{definition}
\noindent Polars have many interesting properties~\cite{ziegler2012lectures}. 
In particular, 
\begin{itemize}
    \item If $\mathcal S_1 \subseteq \mathcal S_2$ then $\mathcal S_2^* \subseteq \mathcal S_1^*$. Moreover, for any bounded $\mathcal S$ its polar $\mathcal S^*$ contains the origin in its interior.

    \item For any invertible matrix $M\in\mathbb R^{d\times d}$, $(M\mathcal S)^* = M^{-\top}\mathcal S^*$. 

    \item Suppose that $\mathcal S$ is a polytope containing the origin in its interior.
 If $\mathcal S$ 
has $r\ge d+1$ vertices, then $\mathcal S^*$ is a polytope with $r$ facets and vice versa. If $\mathcal S$  is also a simplex, that is, an $(r-1)$-dimensional polytope in $\mathbb R^{r-1}$ with $r$ vertices and $r$ facets, then $\mathcal S^*$ is a simplex. 
By extension, given a matrix $P \in \mathbb{R}^{(r-1) \times r}$ whose columns define a simplex containing the origin in its interior, we will refer to its polar matrix as the matrix $\Theta \in \mathbb{R}^{(r-1) \times r}$ whose columns are the vertices of $\conv(P)^*$

    \item For a polytope $\mathcal S$ containing the origin in its interior, $(\mathcal S^*)^* = \mathcal S$.

\item The polar of the unit ball is itself, and for any matrix $Q\in \mathbb R^{r-1\times r}$ {such that $\left[ \begin{array}{c}
                         Q\\e^\top/\sqrt r
                        \end{array} 
                        \right] $ is an $r\times r$ orthogonal matrix}, 
                        the polar of $\conv(Q)$ is $\conv(-rQ)$. 
\end{itemize}

Let us come back to SSMF: given $Y$, we need to find $P$ such that $\conv(Y) \subseteq \conv(P)$. 
In the polar, we will have $\conv(P)^* \subseteq \conv(Y)^*$, where the vertices of $\conv(P)^*$ are the facets of $\conv(P)$, given that the origin belongs to the interior of $\conv(Y)$.  
Hence any matrix $P \in \mathbb{R}^{r-1 \times r}$ such that $\conv(P)^* \subseteq \conv(Y)^*$ corresponds to a feasible solution of SSMF. 
Another well-known property of polars is the following: for a matrix $Y$,  
\begin{align*}
\conv(Y)^* 
&  = \{ \theta \ | \ y^\top \theta \leq 1, y = Yh, h \in \Delta^r \}  \\
&  = \{ \theta \ | \ (Yh)^\top \theta \leq 1, h \in \Delta^r \} \\
&  = \{ \theta \ | \ h^\top (Y^\top \theta) \leq 1, h \in \Delta^r \}   = \{ \theta \ | \ Y^\top \theta \leq e \}, 
\end{align*}
since $h^\top x \leq 1$ for all $h \in \Delta^r$ if and only if $x \leq e$. 
In the following, we will assume that the origin is in the interior of $\conv(P)$. If now $\Theta$ is the polar matrix of $P$, then $\conv(\Theta)^* = \{ x \ | \ \Theta^\top x \leq e \} = \conv(P)$ since the polar of the polar of a polytope is the polytope itself, and the origin is contained in the interior of $\conv(\Theta)=\conv(P)^*$ since $\conv(P)$ is bounded. Given $\Theta$, $P$ can be recovered by computing the vertices of $\conv(\Theta)^* = \{ x \ | \ \Theta^\top x \leq e \} = \conv(P)$, and vice versa.   

The constraint $\conv(\Theta) = \conv(P)^* \subseteq \conv(Y)^*$ can therefore be written as $Y^\top \Theta \leq 1_{n \times r}$ where $1_{n \times r}$ is the matrix of all-ones of size $n$ by $r$.  
Any matrix $\Theta \in \mathbb{R}^{r-1 \times r}$ satisfying $Y^\top \Theta \leq 1_{n \times r}$ and with the origin in the interior of $\conv(\Theta)$  thus corresponds to a feasible solution to SSMF. 

This observation was used in~\cite{abdolali2021simplex} to find a $P$ such that as many data points where located on the facets of $\conv(P)$: 
Let $A = Y^\top \Theta \leq 1_{n \times r}$, then $A(i,j) = 1$ means that the $i$th data points, $Y(:,i)$, is located on the $j$th facet of $\conv(P)$, given by $\{ x \ | \ \Theta(:,j)^\top x = 1\}$, since \mbox{$A(i,j) = \Theta(:,j)^\top Y(:,i) = 1$}. 
Hence maximizing the number of ones in $A$ maximizes the number of data point on the facets of $\conv(P)$, which  has a unique solution (that is, the SSMF is identifiable) under the facet-based condition~\cite{abdolali2021simplex}.

\subsection{Maximizing the volume in the polar} \label{sec:volmax}

In this paper, we do not attempt to maximize the number of data points on the facets of $\conv(P)$, which is a combinatorial problem, which was solved in a greedy fashion using mixed-integer programming via the GFPI algorithm of~\cite{abdolali2021simplex}.   
Instead, we propose to solve the problem at once, maximizing the volume of $\conv(\Theta)$ in the dual space. 
The rationale behind this choice is that the larger a set is, the smallest its dual is, since    
    $\mathcal S_2^* \subseteq \mathcal S_1^*$ 
    implies 
    $\mathcal S_1 \subseteq \mathcal S_2$, and minimizing the volume in the primal has shown to be a powerful approach; see Section~\ref{sec:volmin}.   
We therefore propose to solve the following model: Given $Y \in \mathbb{R}^{r-1 \times n}$, solve 
\begin{equation} \label{eq:firstformupolar} 
    \max_{\Theta \in \mathbb{R}^{r-1 \times r}} 
\vol\big( \conv(\Theta) \big) 
\quad \text{ such that } \quad 
Y^\top \Theta \leq 1_{n \times r}. 
\end{equation}
Recall that the constraint $Y^\top \Theta \leq 1_{n \times r}$ is equivalent to $\conv(\Theta) \subseteq \conv(Y)^*$. 
The volume can be computed as follows 
\begin{align*}
   \vol\big( \conv(\Theta) \big) 
& = \frac{1}{(r-1)!}
\left| \det  \left[ \begin{array}{c}
                        \Theta \\
                        e^\top 
                        \end{array} 
                        \right] \right|.
\end{align*}


\paragraph{Link with volume minimization} 

Solving~\eqref{eq:firstformupolar} is not equivalent to volume minimization in the primal~\eqref{min_vol}. In fact, the problem of maximizing the volume of $\conv(\hat P)^*$ among all the polar sets of the matrices $\hat P\in \mathbb R^{r-1\times r}$ such that $\conv(Y)\subseteq \conv(\hat P)$ is equivalent to 
\begin{equation} \label{eq:equivalent_primal_minvol} 
    \min_{\hat P \in \mathbb{R}^{r-1 \times r}, z \in \mathbb{R}^{r}} 
\vol\big( \conv(\hat P) \big) \prod_{i=1}^r z_i
\quad \text{ such that } \quad 
\conv(Y)\subseteq \conv(\hat P),\quad 
\left[ \begin{array}{c}
                         \hat P\\
    e^\top
                        \end{array} 
                        \right]
 z = e_r, \quad z\ge 0. 
\end{equation}
Notice that $\conv(Y)\subseteq \conv(\hat P)$ can be rewritten as $Y = \hat PH$ where $H(:,j)\in \Delta^r$ for all $j$. We can thus observe that \eqref{eq:equivalent_primal_minvol} differs from \eqref{min_vol} and will in general give different results. The key difference is the presence of the vector $z$, representing the barycentric coordinates of the origin with respect to the simplex whose vertices are the columns of $\hat P$. 
Consider for example a simplex $\hat P$ with a small $z_i$, implying that the origin is very close to one of the facets of $\hat P$. In turn this yields one of the constraint of $\hat P$ to be represented in the dual by a vector $\theta_i$ whose norm is proportional to $1/z_i$ and consequentially very large.  This is the rationale linking the volume of $\Theta$ and the vector $z$.  \\

We will discuss in details how we handle \eqref{eq:firstformupolar} in Section~\ref{sec:optim}, and how we can adapt it in the presence of noise. 
But first, we discuss the identifiability guarantees of solving~\eqref{eq:firstformupolar}. 

\section{Identifiability} \label{sec:identif}

In this section, we prove identifiability of dual volume maximization under various assumptions, namely under the SSC (Section~\ref{sec:ssc}), separability (Section~\ref{sec:separability}), and a new condition between the two which we call $\eta$-expansion (Section~\ref{sec:expanded}). 
As we will show, the identifiability depends on the choice of the translation vector $v$, and we provide in Section~\ref{sec:minmaxapproach} a min-max formulation that optimizes the choice of $v$ (Section~\ref{sec:minmaxapproach}). This will be the formulation we solve in Section~\ref{sec:optim} to tackle SSMF.

\subsection{SSC}   \label{sec:ssc} 

Let $X= WH$ be a rank-$r$ SSMF. 
After the preprocessing discussed in Section \ref{sec:preproc}, we find the corresponding SSMF of $Y= PH$, where now $Y\in \mathbb R^{r-1\times n}$ and $P\in \mathbb R^{r-1\times r}$, with the same matrix $H$. Since the SSC condition in Definition \ref{def:SSC} is tested on the matrix $H$, we can suppose from now on that, equivalently, $X$ or $Y$ has an SSC decomposition. 

Fist of all, we prove that if the translation preprocessing of $X$ is operated with respect to the vector $v$ corresponding to the center of $W$, that is, $v = We/r$, then the matrix $\Theta$ polar of $P$ is the unique solution of the maximization problem \eqref{eq:firstformupolar}.  Recall that $P = U^\top [W - ve^\top]$, so $Pe = 0$. 

In a nutshell, after a preconditioning with the singular values and left singular vectors of $P$, we find that the columns of the matrix $Y$ are included in a regular and centered simplex circumscribed to the unit ball in $\mathbb R^{r-1}$, while preserving $H$ and thus the SSC property. The unit ball is self-polar, so the SSC condition forces any possible point of $\conv(Y)^*$ to lie inside the unit ball. The simple observation that any maximum volume simplex contained in the ball is necessarily regular, and that the regularity is invariant by polar transformation, concludes the proof.

\begin{theorem}
  \label{th:identifiability_for_SSC}  Let $Y$ be an $(r-1)\times n$ real matrix with $n\ge r$  such that $Y = PH$ with $P$ an $(r-1)\times r$ full rank real matrix, $Pe = 0$, and $H$ an $r\times n$  SSC and column stochastic matrix.  Then
  \begin{equation*}
       \max_{\Theta \in \mathbb{R}^{r-1 \times r}} 
\vol\big( \conv(\Theta) \big) 
\quad \text{ such that } \quad 
Y^\top \Theta \leq 1_{n \times r}. 
\tag{\ref{eq:firstformupolar}}
  \end{equation*}
 is uniquely solved by the polar matrix of $P$. 
\end{theorem}
\begin{proof}
 Recall that the problem is equivalent to 
  \[
     \max_{\Theta \in \mathbb{R}^{r-1 \times r}} 
\vol\big( \conv(\Theta) \big) 
\quad \text{ such that } \quad 
 \conv(\Theta) \subseteq \conv(Y)^*. 
  \]
    Let $P = U\Sigma Q$ be the reduced SVD of $P$ where $U\in \mathbb R^{r-1\times r-1}$ is orthogonal, $\Sigma\in \mathbb R^{r-1\times r-1}$ is diagonal and invertible and $Q\in \mathbb R^{r-1\times r}$ {is such that $\left[ \begin{array}{c}
                         Q\\e^\top/\sqrt r
                        \end{array} 
                        \right]  $ is an $r\times r$ orthogonal matrix}, since $Pe=0$.  
                        Calling $\Psi =  \Sigma U^\top\Theta$, the problem transforms into
    \begin{equation}
        \label{eq:preconditioned_maxvol}  \det(\Sigma)^{-1} \max_{\Psi \in \mathbb{R}^{r-1 \times r}} 
\vol\big(  \conv( \Psi) \big) 
\quad \text{ such that } \quad 
 \conv( \Psi) \subseteq \conv(QH)^*. 
    \end{equation}
     Since $H$ is SSC, $\conv(QH) = Q\conv(H)\supset Q(\Delta^r \cap \mathcal C)$ and it is easy to prove that
    $Q(\Delta^r \cap \mathcal C) = \sqrt{\frac 1{r(r-1)}}B^{r-1}$, where $B^{r-1}$ is the $r-1$ dimensional unit ball. The polar of the unit ball is itself, so     \[
    \conv( \Psi) \subseteq \conv(QH)^* \subseteq  \left(\sqrt{\frac 1{r(r-1)}}B^{r-1}\right )^* = \sqrt{r(r-1)}B^{r-1}, 
    \]
and in particular all the columns of $\Psi$ are bounded in squared norm by $r(r-1)$. 
Applying the formula for the volume, we find that
\[
\vol\big(  \conv( \Psi) \big) = \frac 1{(r-1)!}\left|\det 
\left[ \begin{array}{c}
                        \Psi \\e^\top 
                        \end{array} 
                        \right]  \right|
    = \frac {r^{\frac {r-1}2}}{(r-1)!}\sqrt {\det 
    \left[ \begin{array}{c c}
                        \frac 1{\sqrt{r}}\Psi^\top & e
                        \end{array} 
                        \right]  
     \left[ \begin{array}{c}
                        \frac 1{\sqrt{r}} \Psi \\e^\top 
                        \end{array} 
                        \right]  
    } =  \frac 1{(r-1)!}\sqrt {\det R }.
\]
Each element of the diagonal in  $R $ is bounded by $r$, so its trace is at most $r^2$. Since $R$ is positive semi-definite, its determinant is bounded through the arithmetic and geometric means inequality (AM-GM) by  $\det(R) \le [\tr(R)/r]^r\le r^r$, and the equality is attained if and only if $R = rI$ or equivalently when $ \left[ \begin{array}{c}
                         \Psi/r \\e^\top/ \sqrt{r}
                        \end{array} 
                        \right]   $ is orthogonal. The matrix $\Psi = -rQ$ thus attains the maximum possible volume and $\conv(-rQ) = \conv(Q)^* \subseteq \conv(QH)^*$, so it is also a solution of problem \eqref{eq:preconditioned_maxvol}. All other $\Psi$ with the same volume such that $\conv( \Psi) \subseteq  \sqrt{r(r-1)}B^{r-1}$  are rotated versions of $-rQ$, that is, $\hat\Psi = -rVQ$, where $V$ is orthogonal, but
    \[
   \conv(-rVQ) = \conv(VQ)^* \subseteq \conv(QH)^* \implies \conv(QH)\subseteq \conv(VQ)
    \]
    \[
    \implies 
   \conv(H)\subseteq \conv\left(
    \left[ \begin{array}{c c}
                            Q^\top & e/\sqrt r
                        \end{array} 
                        \right]  
                         \left[ \begin{array}{cc}
                           V &\\& 1
                        \end{array} 
                        \right]  
                         \left[ \begin{array}{c}
                         Q\\e^\top/\sqrt r
                        \end{array} 
                        \right]  
   \right) = \conv (\Pi), 
   \]
and from SSC, $\Pi= Q^\top V Q + ee^\top /r$ is necessarily a permutation matrix. The simple observation that $\hat \Psi = -rQ\Pi$ lets us conclude that the only solutions to problem \eqref{eq:preconditioned_maxvol} are $-rQ$ and its permuted versions, or also said all possible polar matrices of $Q$. Tracing back to the original problem, we find that all possible solutions of \eqref{eq:firstformupolar} are the polar matrices of $\Theta^* = (U\Sigma^{-1}\Psi)^* = U\Sigma Q = P$. 
\end{proof}

\subsection{Separability} \label{sec:separability}

When the translation is operated with a vector $v$ different from $We/r$, the SSC property is not enough anymore to guarantee that problem \eqref{eq:firstformupolar} is solved by the polar matrix of $P$. We can thus turn to the stronger separability condition. 
In this case, whenever $v$ is in the relative interior of $\conv(X)= \conv(W)$, then the  problem \eqref{eq:firstformupolar} correctly identifies the sought matrix $\Theta$. 
The idea is very simple:   the separability is invariant by the 
preprocessing of Section~\ref{sec:preproc},   and any feasible $\Theta$ in \eqref{eq:firstformupolar} must satisfy $\conv(\Theta)\subseteq \conv(Y)^* = \conv(P)^*$ and in particular,  the polar set of $\conv(P)$ has volume larger or equal than $\conv(\Theta)$. The only case of equality is for when the columns of $\Theta$ coincide with the vertices of $\conv(P)^*$ in some order. This is enough to prove the following result.

\begin{theorem}
  \label{th:identifiability_for_separable}  Let $Y$ be a $(r-1)\times n$ real matrix with $n\ge r$  and a separable decomposition $Y = PH$ with $P$ an $(r-1)\times r$ real matrix, and $H$ an $r\times n$  column stochastic matrix containing the $r \times r$ identity matrix as a submatrix (see Section~\ref{sec:intro:separability}).  If  $0$ is in the interior of $\conv(Y)$, then
  \begin{equation*}
       \max_{\Theta \in \mathbb{R}^{r-1 \times r}} 
\vol\big( \conv(\Theta) \big) 
\quad \text{ such that } \quad 
Y^\top \Theta \leq 1_{n \times r}. 
\tag{\ref{eq:firstformupolar}}
  \end{equation*}
 is uniquely solved by the polar matrix of $P$. 
\end{theorem}
\begin{proof}
    This follows directly from the fact, under the separability assumption, the solution $\conv(\Theta) = \conv(Y)^*$ is feasible and therefore is the unique solution with  
    maximum volume within $\conv(Y)^*$.  
\end{proof}

Notice that in the separable case, $v = Xe/r$ is already a good choice for the translation vector. In fact, under Assumption~\ref{ass:dimaffhull}, it can be proved that $v$ is in the interior of $\conv(X) = \conv(W)$.

\subsection{Between SSC and separability: $\eta$-expanded}  \label{sec:expanded}   

We have seen that for a SSC decomposition, we need a precise translation in the preprocessing of $X$, and instead in the separable case practically any sensible translation yields the correct solution, and we have a perfect candidate for it. To investigate what happens when the problem is not separable, but is more than SSC, we need to introduce a new concept called \textit{expansion} of the data.

\begin{definition}
\label{def:eta_expanded} We say that $H\in \mathbb R^{r\times n}$ is $\eta$-expanded with $\eta\in [0,1]$ if     \[
  \mathcal H_\eta:= \Delta_r\cap \left\{x\in \mathbb R^r\,\,\Big|\,\, x\le \left[\eta + (1-\eta)\frac 2r\right]e\right\} \subseteq \conv (H).
    \]
\end{definition}

\noindent Suppose that $X\in \mathbb R^{m\times n}$ has rank $r-1$ and admits a decomposition $X = WH$ where $H$ is column stochastic and $\eta$-expanded. The following properties are easily shown:
\begin{itemize}
    \item $\eta=1$ if and only if $X$ is separable,
    \item  if $\eta > 0$ then $H$ is SSC,
    \item  $\mathcal C\subset\cone(\mathcal H_0)$.
\end{itemize}
In other words, $0$-expansion is close to the SSC, and the property of being $\eta$-expanded bridges between SSC and separability. 
The set $\mathcal H_\eta$ can be described also as the intersection of $\Delta^r$ and $\Delta^r_\mu$ obtained by symmetrizing $\Delta^r$ with respect to its center $e/r$ and then expanding it by a constant $\mu = (r-2)\eta +1\in [1,r-1]$, as we can see in Figure \ref{fig:expansion}. In formulae, 
\begin{equation}
    \label{eq:delta_mu} 
    \mathcal H_\eta = \Delta^r\cap \Delta^r_\mu = \conv(I) \cap \conv\left( \frac {\mu +1}ree^\top-\mu I\right).
\end{equation}

\begin{figure*}[!htbp]
	\begin{minipage}[b]{0.5\linewidth}
		\centering
		\centerline{\includegraphics[trim={8cm 3cm 11cm 0cm},clip,width=8cm]{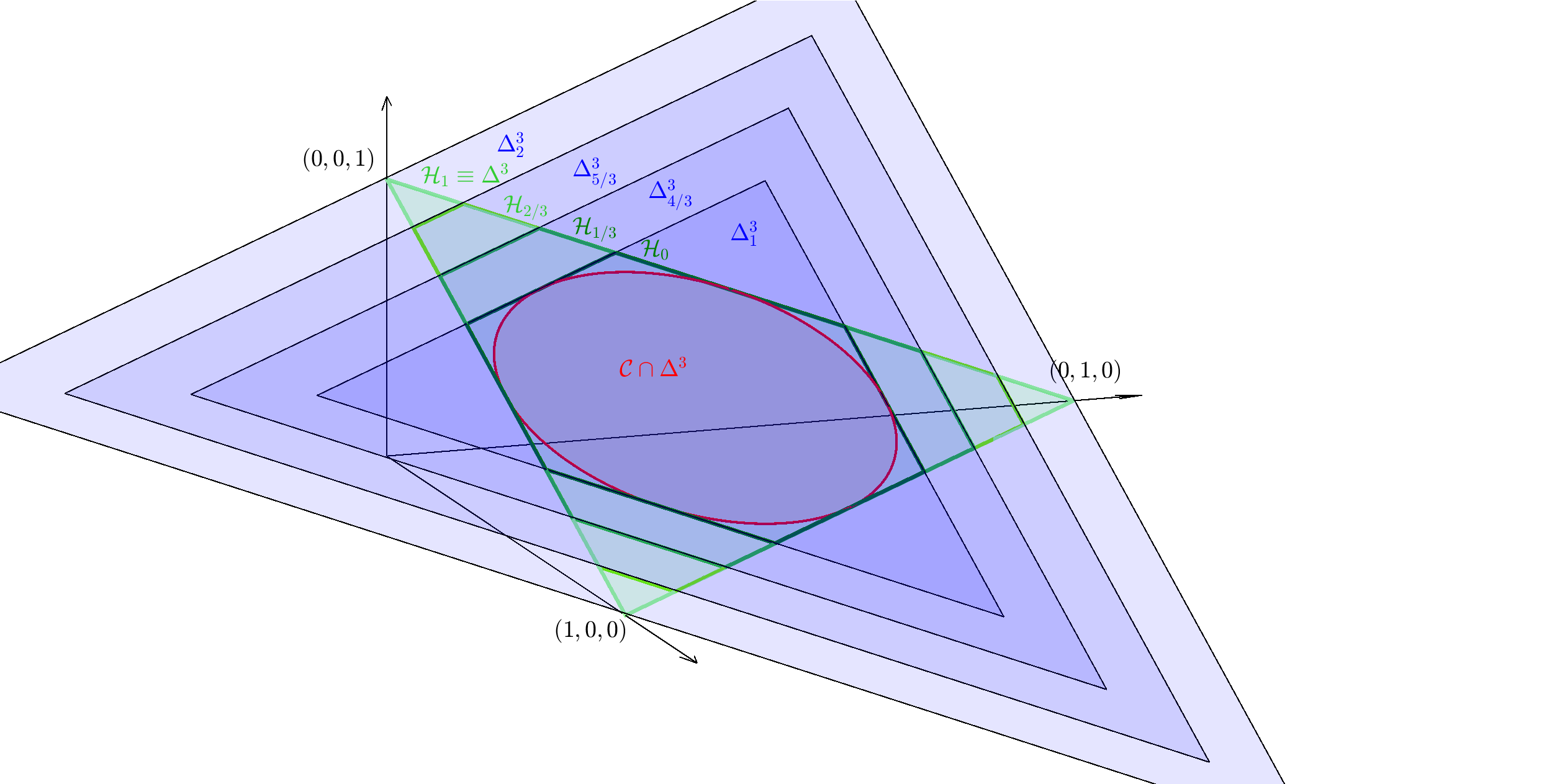}}
		\center{(a) $\Delta_\mu^3$ for $\mu = 1,4/3,5/3,2$ and the associated $\mathcal H_\eta = \Delta^3\cap\Delta_\mu^3$ for $\eta = \mu-1 = 0,1/3,2/3,1$.}\medskip
	\end{minipage}
	\hfill
	\begin{minipage}[b]{0.5\linewidth}
		\centering
		\centerline{\includegraphics[trim={12cm 6cm 12cm 1cm},clip,width=8cm]{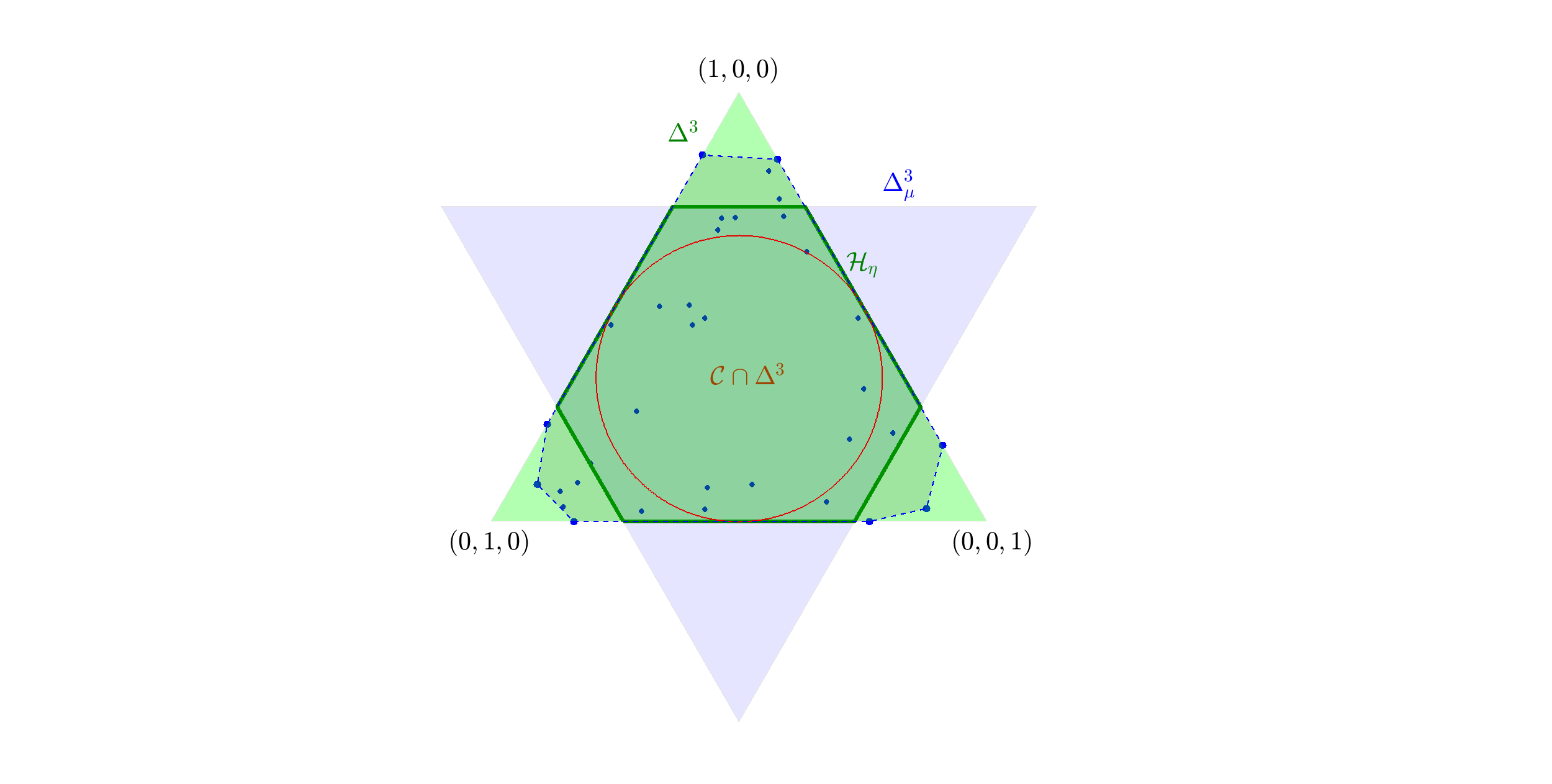}}
		\center{(b) Two-dimensional projection of the columns of a $\eta$-expanded and column stochastic $H$ (dots). }\medskip
	\end{minipage}
	\caption{Visual representation in 3 and 2 dimensions for the unit simplex $\Delta^3$, the cone $\mathcal C$ intersected with $\Delta^3$, the symmetrized and expanded $\Delta_\mu^3$,  the associated $\mathcal H_\eta = \Delta^3\cap\Delta_\mu^3$ and a $\eta$-expanded $H$.}
	
	\label{fig:expansion}
\end{figure*}
\

In case of SSC, Theorem \ref{th:identifiability_for_SSC} tells us that the only certified good translation vector is $v = We/r$. Instead, in case of separability, Theorem \ref{th:identifiability_for_separable} tells us that all vectors inside the interior part of $\conv(W)$ are good, that is, any vector that can be written as $v = Wq$, where $q$ is strictly positive whose entries sum to one.  
When $H$ is column stochastic and $\eta$-expanded, we can prove that any translation vector that can be written as $v = Wq$, where $q$ is strictly positive, whose entries sum to one, and $0<q< \frac{r\eta +2(1-\eta)}{2r}e$,  yields the correct solution to problem~\ref{eq:firstformupolar}.  
To do so, we first need two lemmas that show how the polar duality behaves under translation of the polytopes and how to compute the volume of the polar matrix after such translation. We provide the proofs in Appendix~\ref{app:proofs}. 
From now on, we use $\mathcal S^\circ$ to indicate the interior of a set $\mathcal S$.  
\begin{lemma}
       \label{lem:polar_after_translation} 
Suppose the columns of $\Theta\in \mathbb R^{(r-1)\times r}$ are the vertices of the polar set of  a convex polytope $\mathcal S$.  for any $w\in \mathcal S^\circ$ suppose that $\Theta_w\in \mathbb R^{(r-1)\times r}$  are the vertices of the polar set of $\mathcal S - w$. If $z_w$ is such that  $\Theta_wz_w = 0$ and $e^\top z_w=1$, then the matrix $\Theta_w\diag(z_w)$ does not depend on $w$ and $\Theta_w = \Theta \diag(e-\Theta^\top w)^{-1}$.
       
In particular, given a matrix $A\in \mathbb R^{(r-1)\times r}$ with $Ae =0$,  for any $w\in \conv(A)^{\circ}$ call $\Theta_w$ the polar of $A-we^\top$ and suppose  $At = v$ and $As = z$ with $t,s\in \Delta^r$ and $v,z\in \conv(A)^{\circ}$. Then $ \Theta_v \diag(t)= \Theta_z \diag(s)$ and  $\Theta_vt  = \Theta_zs  =  0$. 
       \end{lemma}

\begin{lemma}
\label{lem:volume_after_multiplication}  Given  $\Theta\in \mathbb R^{(r-1)\times r}$, suppose that $\Theta w=0$ for a nonzero vector $w$ with $e^\top w \ne 0$. Then for any invertible matrix $N$,
\begin{equation}
\label{eq:volume_transformation}
     \vol\big( \conv(\Theta N) \big)  = |\det(N)|\, \left|\frac{e^\top N^{-1}w}{e^\top w}\right|\,   \vol\big( \conv(\Theta) \big). 
\end{equation}
\end{lemma}

Now we can state and prove our result.

\begin{theorem}
 \label{thm:eta_expanded}   Suppose that $Y=PH$  with $H$ $\eta$-expanded and column stochastic and $P$ full rank. Consider the vector $q$ such that $Pq = 0$ and $e^\top q = 1$. 
   If $0<q< \frac{r\eta +2(1-\eta)}{2r}e$, then the problem 
  \begin{equation*}
       \max_{\Theta \in \mathbb{R}^{r-1 \times r}} 
\vol\big( \conv(\Theta) \big) 
\quad \text{ such that } \quad 
Y^\top \Theta \leq 1_{n \times r}. 
\tag{\ref{eq:firstformupolar}}
  \end{equation*}
  is solved uniquely by the polar matrix of $P$.  
\end{theorem}
\begin{proof}
Let\footnote{We abuse notation here since $v$ is now the translation vector in the reduced space, not in the original one.} $v:= Pe/r$ and let $P_v:= P - ve^\top$ with its SVD being $P_v = U\Sigma Q$. 
 Notice that  $Q\in \mathbb R^{r-1\times r}$ {is such that $\left[ \begin{array}{c}
                         Q\\e^\top/\sqrt r
                        \end{array} 
                        \right]  $ is an $r\times r$ orthogonal matrix}, since $P_ve=0$.  
Similarly as  the proof of Theorem \ref{th:identifiability_for_SSC},  problem \eqref{eq:firstformupolar} is equivalent to     \begin{equation}
        \label{eq:preconditioned_maxvol_extended}  \det(\Sigma)^{-1} \max_{\Psi \in \mathbb{R}^{r-1 \times r}} 
\vol\big(  \conv( \Psi) \big) 
\quad \text{ such that } \quad 
 \conv( \Psi) \subseteq \conv(Q(I - qe^\top)H)^*. 
    \end{equation}
where $\Psi =  \Sigma U^\top\Theta$ and
\[
Q(I - qe^\top)H =( Q + \Sigma^{-1}U^\top (ve^\top-P)qe^\top)H
= \Sigma^{-1}U^\top (U\Sigma Q + ve^\top)H = \Sigma^{-1}U^\top Y.
\]
Since $\frac{r\eta +2(1-\eta)}{2r}<\frac{r\eta +2(1-\eta)}{r}$, the vector $q$ is in the interior of $\mathcal H_\eta$ and as a consequence $0=Q(I - qe^\top)q$ is in the interior of $Q(I - qe^\top)\mathcal H_\mu$.  This enables us to freely utilize the properties of the polar duality and find the necessary condition $\conv( \Psi) \subseteq (Q(I - qe^\top)\mathcal H_\eta)^* = \conv\big ((Q(I - qe^\top)\Delta^r)^*\cup(Q(I - qe^\top)\Delta^r_\mu)^*\big)$ where $\Delta_\mu^r$ is defined in \eqref{eq:delta_mu} and $\mu = (r-2)\eta +1\in [1,r-1]$. The vertices of the polytope $(Q(I - qe^\top)\mathcal H_\eta)^*$ are thus (contained in the set of)  the vertices of $(Q(I - qe^\top)\Delta^r)^*= \conv(Q - Qqe^\top)^* = \conv(M)$ and of 
$$(Q(I - qe^\top)\Delta^r_\mu)^* =   -\frac 1\mu \conv \left(Q(I - qe^\top)\left(   I - \frac {\mu +1}{r\mu }ee^\top\right)\right)^* =  -\frac 1\mu \conv\left(Q + \frac {1}{\mu }Qqe^\top\right)^*.$$ 
Notice that $-\frac {1}{\mu }Qq = Q\left(-\frac q\mu + \frac{\mu+1}\mu \frac er\right)$, so due to Lemma \ref{lem:polar_after_translation}, we get 
\begin{align*}
   (Q(I - qe^\top)\Delta^r_\mu)^* &= -\frac 1\mu \conv\left(Q - Q\left(-\frac q\mu + \frac{\mu+1}\mu \frac er\right)e^\top\right)^* 
    = 
    \conv\left(   M
   \diag \left( \frac{1}{1 - \frac{\mu+1}{rq_i} }\right)
    \right).
\end{align*}
The vertices of a maximum volume simplex $\Psi$ in $(Q(I - qe^\top )\mathcal H_\eta)^*$ must correspond to $r$ of its vertices, so from the above computation can only be a column $m_i$  of $M$ or  $\alpha_i m_i$, where $\alpha_i = 1/(1 - \frac{\mu+1}{rq_i})$. If $\Psi$ has among its vertices $\{m_i,\alpha_im_i,m_j,\alpha_jm_j\}$ with $i\ne j$, then the rank of $\Psi$ is at most $r-2$ and its volume is zero. Therefore, we only need to consider the following sets of $r$ vertices $\{v_1,\dots,v_r\}$:
\begin{enumerate}
    \item $v_i \in \{1,\alpha_i  \} m_i$ for all $i$
    \item there exists exactly one index $i$ such that both $\alpha_im_i$ and $m_i$ are among the vertices. 
\end{enumerate}
Since $M$ is the polar of $Q - Qqe^\top$, Lemma    \ref{lem:polar_after_translation} says that $Mq = 0$. As a consequence, using \eqref{eq:volume_transformation},  the volume of any simplex of the first kind is 
\[
V_1 := \left|\prod_{i\in S} \alpha_i\right | \left|1 + \sum_{i\in S} q_i\left(\frac {1}{\alpha_i} -1\right) \right | \vol(\conv(M)) 
= 
\left|\prod_{i\in S}  \frac{1}{1 - \frac{\mu+1}{rq_i} } \right | \left|1 - |S|  \frac{\mu+1}{r} \right | \vol(\conv(M)), 
\]
where $S:=\{ i \,|\, v_i = \alpha_im_i\}$ and if $S$ is empty then $V_1$ is equal to $\vol(\conv(M))$. By hypothesis, $q_i< \frac{r\eta +2(1-\eta)}{2r} = \frac{\mu +1}{2r}$, so $\alpha_i<1$. As a consequence, if $|S|(\mu+1)\le 2r$ and $S$ is not empty, we find that $V_1<\vol(\conv(M))$. For   $|S|(\mu+1)> 2r$, we have 
\[
V_1 =
\prod_{i\in S}  \frac{1}{ \frac{\mu+1}{rq_i} - 1 }  \left( |S|  \frac{\mu+1}{r} -1 \right ) \vol(\conv(M)), 
\]
but thanks to Jensen Inequality applied to the concave function $f(x)=\ln{\frac 1{1/x-1}}$ with weights equal to $1/|S|$ and points $x_i= rq_i/(\mu +1)<1/2$ we get
\begin{align*}
    \prod_{i\in S}\frac {1}{   \frac {\mu +1}{rq_i} -1 }
    &= \exp\left(\sum_{i\in S} \ln{\frac {1}{  \frac{\mu +1} {rq_i} -1 } 
    }\right)
    \le  \exp\left(|S| \ln{\frac {1}{   \frac 1{\frac r{|S|(\mu +1)} \sum_{i\in S}  q_i }-1 } 
    }\right)
    \le 
    \left(  \frac {1}{  |S| \frac {\mu +1}{r }-1 }    \right)^{|S|}, 
\end{align*}
and since $|S|>2r/(\mu +1)\ge 2$, we find again that $V_1<\vol(\conv(M)) $.

For the polytopes of the second kind, suppose without loss of generality that $v_1 = m_2$, $v_2=\alpha_2m_2$ and $v_i = \nu_im_i$ for $i>2$, where $\nu_i\in \{1,\alpha_i\}$. Then
\begin{align*}
  V_2 :=   \vol
    \begin{pmatrix}
        m_2& m_2\alpha_2&m_3\nu_3&\dots&m_r\nu_r
    \end{pmatrix} 
    = \frac {\left| \alpha_2 -1 \right| \cdot 
     \left|\prod_{k\ge 3} \nu_k  \right|}
     {(r-1)!} 
     |\det(\hat M)|, 
\end{align*}
where $\hat M$ is the top-right $(r-1)\times (r-1)$ submatrix of $M$. Since $M$ is the polar of $Q(I - qe^\top )$, by Lemma \ref{lem:polar_after_translation} we find that $M = -Q\diag(q)^{-1}$, and if $\hat Q$ is the submatrix of $Q$ associated to $M$, then $ |\det(\hat M)| =  |\det(\hat Q)|/\prod_{i>1}q_i$ 
and by \eqref{eq:volume_transformation}, 
\[\vol(\conv(M)) = \frac {\vol(\conv(Q))}{r\prod_{i}q_i}
= \frac{1}{(r-1)!}\frac {r |\det(\hat Q)|}{r\prod_{i}q_i}
=\frac{|\det(\hat M)|}{(r-1)!}\frac {1}{q_1}.
\]
If now $S:=\{ i \,|\, v_i = \alpha_im_i,\, i>2\}=\{ i \,|\, \nu_i = \alpha_i,\, i>2\}$, then $V_2$ reduces to
\[
V_2 =     \frac {({\mu+1})q_1}{({\mu+1}) - rq_2 }    
     \prod_{i\in S}  \frac{1}{ \frac{\mu+1}{rq_i} - 1 }
    \vol(\conv(M)), 
\]
 but from the hypothesis $q_i< \frac{\mu +1}{2r}$, so it is immediate to see that 
\[
  \frac {({\mu+1})q_1}{({\mu+1}) - rq_2 }   
     \prod_{i\in S}  \frac{1}{ \frac{\mu+1}{rq_i} - 1 }
     <
  \frac {(\mu +1)\frac{\mu +1}{2r} }
     {  (\mu +1) - r\frac{\mu +1}{2r} } =
     \frac{\mu +1}{r}\le 1, 
\]
and thus $V_2 <  \vol(\conv(M))$.

The polytope with the biggest volume inside of  $(Q(I - qe^\top)\mathcal H_\eta)^*$ thus coincides with the polar of $\conv(Q(I - qe^\top))$ that is in particular contained in $\conv(Q(I - qe^\top)H)^*$. The matrices $\Psi$ describing the polar of  $Q(I - qe^\top)$ are therefore the unique solutions to \eqref{eq:preconditioned_maxvol_extended}. Going back to the original problem, we see that it is solved uniquely by $\Theta = U\Sigma^{-1}\Psi$ being the polar of $U\Sigma Q(I - qe^\top) = P_v (I - qe^\top) = P$. 
\end{proof}

When $X$ is separable, that is, $H$ is $1$-expanded, Theorem \ref{th:identifiability_for_separable} says that the only condition needed for the correctness of the solution of the problem \ref{eq:firstformupolar} is   $v = Wq$, where $q$ has sum 1 and $0<q<e$. In this case, though, Theorem \ref{thm:eta_expanded} only holds for $0<q<e/2$. This suggests that the result can be improved.
\begin{conj}
    The thesis of Theorem \ref{thm:eta_expanded} holds if $q_i> \frac{1-\eta }{r}$ for every $i$.
\end{conj}


\subsection{Min-max approach under the SSC} \label{sec:minmaxapproach}

Under the SSC condition, we have proved that the solution to problem \eqref{eq:firstformupolar} coincides with the SSC decomposition $Y = PH$ when  $Pe=0$. 
In the case that $v:=Pe/r\ne 0$  one would need to translate $Y$ by $v$ before solving problem \eqref{eq:firstformupolar}, so that $Y - ve^\top = (P-ve^\top) H$ and the resulting solution $\Theta$ would coincide with the polar set of $P-ve^\top$.  
Since $v$ is not generally known beforehand, we inquire what happens when we translate by a different vector $w$. We find that that the solution $\Theta_w$ of \eqref{eq:firstformupolar} applied to the matrix $Y-we^\top$ has always a strictly larger volume than the correct solution $\Theta\equiv \Theta_v$, and the volume of $\Theta_w$ is actually a convex function in $w$. 


\begin{theorem} \label{th:minmaxformu}
    Let $Y$ be an $(r-1)\times n$ real matrix with $n\ge r$  such that $Y = PH$ with $P$ an $(r-1)\times r$ full rank real matrix and $H$ an $r\times n$  SSC and column stochastic matrix.  If     \begin{equation}
       \mathcal V(w):=\sup_{\Theta \in \mathbb{R}^{r-1 \times r}} 
\vol\big( \conv(\Theta) \big) 
\quad \text{ such that } \quad 
(Y-we^\top)^\top \Theta \leq 1_{n \times r}, 
\label{eq:translated_polar}
  \end{equation}
  for any vector $w\in \mathbb R^{r-1}$, then $\mathcal V(w)$ is a convex function with unique minimum at $w = v = Pe/r$.  
\end{theorem}
\begin{proof}
If $w\not\in \conv(Y)^\circ$, then $\conv(Y)^*$ is unbounded and $\mathcal V(w)=\infty$, so from now on we suppose   $w\in \conv(Y)^\circ\subseteq \conv(P)^\circ$. We can now rewrite the problem as 
  \[
    \mathcal V(w) = \sup_{\Theta \in \mathbb{R}^{r-1 \times r}} 
\vol\big( \conv(\Theta) \big) 
\quad \text{ such that } \quad 
 \conv(\Theta) \subseteq \conv(Y-we^\top)^*. 
  \]
The polar matrix $\Psi_w$ of $Y-we^\top$ represents a polytope, so $\mathcal V(w)$ will be the volume of a simplex $\conv(\Theta_w)$ whose vertices are a $r$-subset of the $s$ columns of $\Psi_w$, as we show in Lemma~\ref{ass:lemsolvert} in the Appendix. 
The maximum is thus achieved by one out of $\binom{s}{r}$  simplices $\Theta_{w}^{(i)}$, 
and we can recast the problem as    
 \[
    \mathcal V(w) = \max_{i=1,\dots,\binom sr} 
\vol\big( \conv(\Theta_w^{(i)}) \big)
\quad \text{ such that } \quad 
\Theta_w^{(i)} = \Psi_w I^{(i)}_{s\times r}, 
  \]
where $\{I^{(i)}_{s\times r}\}_{i=1,\dots,\binom sr}$ are all the possible full rank, binary and column stochastic matrices of size $s\times r$. Since each $\Theta_w^{(i)}$ represents $r$ linear constraints of $\conv(Y-we^\top)^*$, then its polar set $\mathcal S_w^{(i)}$ is just the $w$-translated of a fixed  
(and possibly unbounded) polytope with $r$ facets containing $\conv(Y)$. If we now fix the vector $\ell = Ye/n\in \conv(Y)^\circ$, then by Lemma \ref{lem:polar_after_translation},
\[
 \mathcal V_i(w):=\vol(\conv(\Theta_w^{(i)})) = \vol(\conv(\Theta_\ell^{(i)})\diag(e-(\Theta_\ell^{(i)})^\top (w-\ell))^{-1}) = \frac{\vol(\conv(\Theta_\ell))}{\prod_j [e - (\Theta_\ell^{(i)})^\top(w-\ell)]_j} . 
\]
Notice now that $-\ln(x)$ and $e^x$ are both convex functions, so we can prove that  $\mathcal V_i(w)$ is also a convex function. In fact $[e - (\Theta_\ell^{(i)})^\top(w-\ell)]_j> 0$  for any $w\in \conv(Y)^\circ$ and any $j$,  so for any $\lambda\in [0,1]$ and any couple of points $w_1,w_2\in \conv(Y)^\circ$, 
\begin{align*}
    \mathcal V_i&(\lambda w_1 + (1-\lambda) w_2) =  
    \frac{\vol(\conv(\Theta_\ell))}{\prod_j [e - (\Theta_\ell^{(i)})^\top(\lambda w_1 + (1-\lambda) w_2-\ell)]_j}\\
    &=
    \frac{\vol(\conv(\Theta_\ell))}{\prod_j \lambda [e - (\Theta_\ell^{(i)})^\top( w_1 -\ell)]_j
 +  (1-\lambda) [e - (\Theta_\ell^{(i)})^\top( w_2-\ell)]_j}\\
  &\le  \vol(\conv(\Theta_\ell)) \exp\left(
 -\lambda\sum_j \ln\left(
  [e - (\Theta_\ell^{(i)})^\top( w_1 -\ell)]_j\right) 
  -(1-\lambda)\sum_j \ln\left(   [e - (\Theta_\ell^{(i)})^\top( w_2-\ell)]_j
 \right) 
 \right)\\
  &\le  \vol(\conv(\Theta_\ell)) \left(
 \lambda 
 \prod_j \frac 1{[e - (\Theta_\ell^{(i)})^\top( w_1 -\ell)]_j}
  +(1-\lambda) 
  \prod_j \frac 1{[e - (\Theta_\ell^{(i)})^\top( w_2 -\ell)]_j}
 \right) \\
 &=  \lambda \mathcal V_i( w_1) 
 +  (1-\lambda) \mathcal V_i( w_2).
\end{align*}
The function $\mathcal V(w)$ is now the maximum of convex functions, so it is also convex.

Since  $Y - we^\top = (P-we^\top) H$, we have that the polar matrix $\tilde\Theta_w$ of $P-we^\top$ satisfies \eqref{eq:translated_polar}, so by Lemma \ref{lem:polar_after_translation} and  \eqref{eq:volume_transformation}, 
$$
\mathcal V(w)\ge \vol(\conv(\tilde\Theta_w)) =\vol(\conv(\tilde\Theta_v\diag(1/rt_i))) 
=  \frac {\vol(\conv(\tilde\Theta_v))}{r^r\prod_i t_i},   
$$
where $Pt = w$ and $t\in \Delta^r$. A simple application of AM-GM tells us that  $\prod_i t_i \le 1/r^r$. 
We know by Theorem \ref{th:identifiability_for_SSC} that $\mathcal V(v) = \vol(\conv(\tilde\Theta_v))$, so we conclude that 
    $$\mathcal V(w)\ge  \frac {\vol(\conv(\tilde\Theta_v))}{r^r\prod_i t_i} \ge \mathcal V(v)$$
with equality only if $t_i = 1/r$ for every $i$, i.e., $w = Pe/r=v$.  
\end{proof}

Theorem~\ref{th:minmaxformu} and Theorem \ref{th:identifiability_for_SSC} imply the following.  

\begin{corollary} \label{cor:minmax}
    Let $Y$ be an $(r-1)\times n$ real matrix with $n\ge r$  such that $Y = PH$ with $P$ an $(r-1)\times r$ full rank real matrix and $H$ an $r\times n$  SSC and column stochastic matrix. Then 
    \begin{equation}
      \inf_{w\in\mathbb R^{r-1}} \sup_{\Theta \in \mathbb{R}^{r-1 \times r}} 
\vol\big( \conv(\Theta) \big) 
\quad \text{ such that } \quad 
(Y-we^\top)^\top \Theta \leq 1_{n \times r}, 
\label{eq:minmax_volume}
  \end{equation}
  is solved uniquely by $w =  Pe/r$ and $\Theta$ being the polar matrix of $P-we^\top$.  
\end{corollary}

Corollary~\ref{cor:minmax} incites us to update the translation vector $v$ and the solution $\Theta$ using a min-max approach: $v$ should be chosen to minimize the volume, while $\Theta$ to maximize it. This is described in the next section. 
It is interesting to note that the min-max approach would converge in one iteration under the separability condition (since any $w$ in the convex hull of $Y$ leads to the sought $\Theta$; see Theorem~\ref{th:identifiability_for_separable}), while the set of $v$'s that lead to identifiability typically contains more than the point $Pe/r$, as shown in Theorem~\ref{thm:eta_expanded} when $H$ is $\eta$-expanded. In practice, we will see that alternating minimization of $v$ and $\Theta$ typically converges within a few iterations.

\section{Optimization} \label{sec:optim}

Let us first assume that the translation vector, $v$, is fixed. 
In the presence of noise, we propose to consider the following 
formulation: 
\begin{equation}\label{eq:thirdformupolar}
	\max_{Z,\Theta, \Delta}  \quad \det(Z)^2 - \lambda \|\Delta\|_F^2 \quad 
	\text{ such that }  \quad  
 Z = \left[ \begin{array}{c}
                        \Theta \\
                        e^\top
                        \end{array} \right] 
                        \text{ and }
	 \quad Y^\top \Theta \leq {1_{n \times r}} + \Delta. 
\end{equation}
The matrix $\Delta$ belongs to $\mathbb{R}^{(r-1) \times n}$ and represents the noise matrix, while $\lambda >0$ serves as a regularization parameter. Moreover, we have squared the volume of $\conv (\Theta)$ in the objective to make it smooth (getting rid of the absolute value). 
In this problem, the objective function is nonconcave, however, all the constraints are linear. Inspired by the work of \cite{huang2019detecting}, we use the block successive upperbound minimization (BSUM) framework~\cite{razaviyayn2013unified} and iteratively update the columns $\Theta$.

Using the co-factor expansion within Laplace formula, we express $\det(Z)$ as a linear function of the entries in any $k$-th column:
$\det(Z) = \sum_{j=1}^r (-1)^{j+k}Z(j,k) \det(Z_{-j,-k})$, 
where $Z_{-j,-k}$ is obtained by removing the $j$-th row and $k$-th column from $Z$. 
If we fix all columns of $Z$ but the $k$th, we have 
\[
\det(Z) = {f^{(k)}}^\top Z(:,k), 
\text{ where }  
{f^{(k)}}(j) = (-1)^{j+k} \det(Z_{-j,-k}) 
\quad \text{ for } j=1,\dots,r.
\]
For simplicity, let us denote $c = {f^{(k)}}$ and $x = Z(:,k)$. We want to maximize 
$f(x) = \det(Z)^2 = \big( {f^{(k)}}^\top Z(:,k)\big)^2 = (c^\top x)^2$. The function $f(x) = x^\top (cc^\top) x$ is a convex quadratic function that can be lower bounded by its first-order Taylor approximation, that is, for any $x_0$, 
\[
f(x) = (c^\top x)^2 \geq f(x_0) + \nabla f(x_0)^\top (x-x_0) 
= 
 2 \big[cc^\top x_0 \big]^\top x + 
(c^\top x_0)^2 
= 
d^\top x + \text{ constants}, 
\]
since $\nabla f(x_0) = 2 (cc^\top) x_0$, where 
$d = 2 cc^\top x_0^\top = 2 {f^{(k)}} {f^{(k)}}^\top x_0 = \alpha {f^{(k)}}$, 
where $x_0$ is the previous value of $Z(:,k)$ (from previous iteration), and $\alpha = 2 {f^{(k)}}^\top x_0 = 2 \det(Z)$. Hence we have a ``minorizer'' of $\det(Z)^2$ as a function of $x=Z(:,k)$ around $x_0$. 

Per this inequality, the iterative maximization of $\det(Z)^2$ involves sequentially updating columns of $Z$ and optimizing the lower-bound expression for each column of $Z$ until convergence. 
In each iteration, individual columns of $Z$ (and $\Theta$) are updated by considering every other column as fixed and solving a quadratic programming problem of the form (for $k=1,..,r$):
\begin{equation} \label{eq:optmprob1}
	\max_{t,\Theta(:,k), \Delta(:,k)} \alpha {f^{(k)}}^\top t - {\lambda}  ||\Delta(:,k)||_2^2  \quad \text{ such that }   \quad t = \left[ \begin{array}{c}
                        \Theta(:,k) \\
                        1
                        \end{array} \right] \text{ and } 
	  Y^\top \Theta(:,k) \leq {1_{(r-1) \times 1}}+\Delta(:,k). 
\end{equation}
However, this optimization problem alone is insufficient to guarantee the boundedness of the corresponding simplex in the primal space. The columns in $\Theta$ define a bounded simplex in $\mathbb{R}^{r-1}$ if and only if the positive hull of $\Theta$ spans $\mathbb{R}^r$, or equivalently if $0$ is in the interior of its convex hull.  Consequently, we add the constraint to the problem above 
 $\Theta(:,k)=-\sum_{i \neq k} \alpha_i \Theta(:,i)$ with $\alpha_i \geq \epsilon$     
for some small $\epsilon > 0$. We will use $\epsilon = 0.01$. 


Similar to \cite{huang2019detecting}, we use a numerical trick to define the vector ${f^{(k)}}$ as the columns of $Z^{-1}$. This is based on Carner's rule and helps to avoid round-off errors. 


 \paragraph{Initializing and updating the translation vector $v$} 

 As explained in details in Section~\ref{sec:identif}, the choice of the translation vector $v$ in the preprocessing step, $Y = U^\top (X - ve^\top)$, is crucial for the identifiability of SSMF via volume maximization in the dual. The best choice for $v$ is $We/r$ but it is unknown a priori. 
To initialize $v$, we resort to two strategies: 
\begin{enumerate}
    \item $v_0 = Xe/n$ which is the sample average. This solution could be a bad approximation of $We/r$ when the samples are not well scattered within $\conv(W)$. 

    \item $v_0$ is the average of the vertices extracted by SNPA, an effective separable NMF algorithm. This approach is less sensitive to imbalanced distributions within $\conv(W)$. 
    
\end{enumerate} 
Since the optimal vector $v = We/r$ leads to the smallest volume solution (Theorem~\ref{th:minmaxformu}), we resort to a min-max approach: once~\eqref{eq:optmprob1} is solved and a solution $\Theta$ is obtained, $W$ can be estimated via the vertices of the dual of $\Theta$, by solving a system of linear equations: to estimate the $k$th column of $W$, solve $\Theta(:,j) \hat W(:,k) = 1$ for $j \neq k$ and then let $\tilde{W}(:,k) = U \hat W(:,k) + v$. Then the new translation vector $v$ is chosen as $\tilde{W} e/r$ which minimizes the volume of $\Theta$. 

\paragraph{Mitigating sensitivity to initialization} 

Our numerous numerical experiments have shown that solving the optimization problem in \eqref{eq:optmprob1} is  usually not too sensitive to the initialization. However when there exist two of more candidate simplices with close volumes, the algorithm might converge to suboptimal solutions. 
To reduce sensitivity to initialization, the optimization algorithm is executed multiple times concurrently, each time with distinct random initializations. The selected $\Theta$ is the one that results in the largest volume. We will use five random initializations for this purpose in our numerical experiments.  \\

Algorithm~\ref{algo1} summarizes our proposed algorithm for SSMF, which we refer to as MV-Dual.  


\begin{algorithm}[ht!]
	\caption{Maximum Volume in the Dual (MV-Dual)} \label{algo1}
	\begin{algorithmic}[1]
		\REQUIRE Data matrix $X \in \mathbb{R}^{m \times n}$, a factorization rank $r$, the regularization parameter $\lambda>0$, the number of random initializations $n\_init$ (default = 5). 
		
		\ENSURE A matrix $W$ such that $X \approx WH$ where $H$ is column stochastic.  \vspace{0.2cm} 
		
		{\emph{\% Step 1. Initialization of $v$ and $Y$}} 
  
		\STATE Initialize $v_0$ with the sample mean $v_0 = Xe/n$ {or with $X(:,\mathcal{K})e/|\mathcal{K}|$ where $\mathcal{K}$ is obtained via SNPA}. 
  
		
            \STATE Let $Y = U^\top \left( X - v_0 e^\top \right) 
            = U^\top  X - U^\top v_0 e^\top$, 
            
            where the columns of $U$ are the first $(r-1)$ singular vectors of $X-v_0 e^\top$.  \vspace{0.2cm}  
            

		{\emph{\% Step 2. Initialize the set of solutions}}

            \STATE Initialize the set of $n\_init$ solutions as  
            $\mathcal{S}=\{ Z_i \}_{i=1}^{n\_init}$ where $Z_i =\left[ \begin{array}{c}
                        \ {\Theta}_i \\
                        e^\top
                        \end{array} \right] \in \mathbb{R}^{r \times r}$ and the entries of ${\Theta}_i\in \mathbb{R}^{r-1 \times r}$ are sampled from $\mathcal{N}(0,1)$.  
            \STATE $p = 1$. 
            
            \WHILE{not converged: $p=1$ or $\frac{\|v_p-v_{p-1}\|_2}{\|v_{p-1}\|_2} > 0.01$} 
                \STATE{\emph{\% Step 3.a. Update $\Theta$ and $W$}} 
                
                \FOR{each candidate matrix $Z_i$ in $\mathcal{S}$ (can be parallelized)} 
                		\STATE Solve \eqref{eq:thirdformupolar} via alternating optimization to update $Z_i$ and $\Theta_i$.  \label{algo:stepopti8}
                \ENDFOR
                \STATE Compute the volume of each of candidate solutions in $\mathcal{S}$ and select the one with the largest volume, which we denote $\Theta$. 

       \STATE Recover $\hat{W}$ by computing the dual of  $\conv\left( \Theta\right)$.

        \STATE Project back to the original space: $W = U\hat W + v_{p-1}$.

        
                \STATE{\emph{\% Step 3.b. Update $v$ and $Y$}}

                \STATE Let $v_p \leftarrow W e / r$, and 
                let $Y \leftarrow U^\top X - U^\top v_p e^\top$. 

                
                \STATE $p = p+1$. 
            \ENDWHILE

	\end{algorithmic}
\end{algorithm}

\paragraph{Computational cost} 

The preprocessing requires the computation of the truncated SVD, in $\mathcal{O}(mnr^2)$ operations.  
The main cost is to solve~\eqref{eq:thirdformupolar} by alternatively optimizing~\eqref{eq:optmprob1} which is a quadratic program in $\mathcal{O}(n)$ variables and constraints. Such problems require $\mathcal{O}(n^3)$ operations in the worst case. However, we have observed that it is typically solved significantly faster by the solver; rather in linear time in $n$ --we will solve real instances of~\eqref{eq:thirdformupolar} with $n = 10^4$ in 15 seconds (Table~\ref{jaspert}). The reason is that this problem has a particular structure. 
The variables $Z(:,k)$ and $\Theta(:,k)$ are $r$-dimensional, while the $n$-dimensional variable, $\Delta(:,k)$, only appears with the identity matrix in the constraints. In the noiseless case, $\Delta(:,k)=0$ and hence it could be removed from the formulation leading to a 
$\mathcal{O}(r^3)$ complexity. 
In the noisy case, only a few entries of $\Delta(:,k)$ will be non-zero, namely the entries corresponding to data points outside the hyperplane defined by $\Theta(:,k)$. Further research include the design of a dedicated solver to tackle~\eqref{eq:optmprob1}, e.g., using an active-set approach. 

Note that in step~\ref{algo:stepopti8} of Algorithm~\ref{algo1}, we use the stopping criterion $\frac{\|Z_\ell - Z_{\ell-1}\|_F}{\|Z_{\ell-1}\|_F} \leq 10^{-3}$ where $Z_\ell$ is obtained after updating each column of $Z_{\ell-1}$ using~\eqref{eq:optmprob1}, or a maximum number of 100 iterations.

\section{Numerical experiments} \label{sec:exp}

In this section, we present numerical experiments to show the efficiency of the proposed MV-Dual algorithm under various settings and conditions. All experiments are implemented in Matlab (R2019b), and run on a laptop with Intel Core i7-9750H, @2.60 GHz CPU and 16 GB RAM. The code, data and all experiments are available from~\url{https://github.com/mabdolali/MaxVol_Dual/}.  


\paragraph{SSMF algorithms} We compare the performance of MV-Dual to six state-of-the-art algorithms:
\begin{itemize}
	\item Successive nonnegative projection algorithm (SNPA)~\cite{gillis2014successive} is based on \textit{separability} assumption and presents a robust extension to the successive projection algorithm (SPA)~\cite{araujo2001successive, gillis2014successive} by taking advantage of the nonnegativity constraint in the decomposition.
 
	\item Simplex volume minimization (Min-Vol) fits a simplex with minimum volume to the data points using the following optimization problem~\cite{leplat2019minimum}:
	\[\min_{W,H} \|X-WH\|_F^2 + \lambda \logdet(W^\top W + \delta I_r) \quad \text{s.t. } H(:,j) \in \Delta^r \ \text{for all } j.\]
	This problem is optimized based on a block coordinate descent approach using the fast gradient method. The parameter $\lambda$ is chosen as in~\cite{leplat2019minimum}: $\tilde{\lambda} \frac{\|X-W_0 H_0\|_F^2}{\logdet(W_0^\top W_0 + \delta I_r)}$ where $(W_0,H_0)$ is obtained by SNPA and $\tilde{\lambda} \in \{0.1,1,5\}$ where 0.1 is the default value in~\cite{leplat2019minimum}.  
 
        \item Minimum-Volume Enclosing Simplex (MVES)~\cite{chan2009convex} searches for an enclosing simplex with minimum volume and converts the problem into a determinant maximization problem by focusing on the inverse of $\Tilde{W}$ defined in \eqref{eq:tW}. 
        
	\item Maximum volume inscribed ellipsoid (MVIE)~\cite{lin2018maximum} inscribes a maximum volume ellipsoid in the convex hull of the data points to identify the facets of $\conv(W)$. 
 
	\item Hyperplane-based Craig-simplex-identication (HyperCSI)~\cite{lin2015fast}: HyperCSI is a fast algorithm based on SPA but does not rely on separability assumption. HyperCSI extracts the \textit{purest} samples using SPA and uses these samples to estimate the enclosing facets of the simplex.
 
	\item Greedy facet based polytope identification (GPFI)~\cite{abdolali2021simplex} has the weakest conditions to recover the unique decomposition among the stat-of-the-art methods. This approach sequentially extracts the facets with largest number of points by solving a computationally expensive mixed integer program. 
\end{itemize}

To assess the quality of a solution, $W$, we measure the relative distance between the column of $W$ and the columns of the ground-truth $W_t$: 
	\[
 ERR = \min_{\pi, \text{a permutation}} \frac{\|W_t - W_\pi\|_F}{\|W_t\|_F},
 \]
	where $W_\pi$ is obtained by permuting the columns of $W$.

\subsection{Synthetic data} 

In this section, we compare the SSMF algorithms on noiseless and noisy synthetic data sets.



\paragraph{Data generation}

We generate synthetic data following~\cite{abdolali2021simplex}. Two categories of samples are generated: $n_1$ samples are produced exactly on the $r$ facets, and $n_2$ samples are produced within the simplex, for a total number of $n = n_1 + n_2$ samples. The entries of the ground-truth matrix $W_t$ are uniformly distributed in the interval $[0,1]$ and the non-zero columns of $H_t$ are generated using the Dirichlet distribution with all parameters equal to $1/d$ where $d$ is the dimension of the simplex where samples are generated. 
We define the {purity} parameter $p \in (\frac{1}{r-1},1]$ as  $p(H_t) = \min_{1 \leq k \leq} \|H_t(k,:)\|_{\infty}$ which quantifies how well the ground-truth data is spread within $\conv(W_t)$. (The lower bound $\frac{1}{r-1}$ comes from the fact that $n_1$ columns of $H$ are on facets of $\Delta^r$, that is, have at least one entry equal to zero.) 
Given a purity level $p$, columns of $H$ are resampled as long as they contain an entry larger than $p$. 
Note that the separability assumption is satisfied when $p(H_t) = 1$, hence the columns of $W_t$ appear as columns among the samples in $X$. 
The SSC condition is satisfied for smaller values of purity values~\cite{lin2018maximum}. 
For the noisy setting, we add independent and identically distributed mean-zero Gaussian noise to the data, with variance chosen according to the following formula for a given signal-to-noise (SNR) ratio: 
\[
\text{variance} = \frac{\sum_{i=1}^m \sum_{j=1}^n X(i,j)^2}{10^{SNR/10} \times m \times n}. 
\]

\paragraph{Parameter setting} {For noiseless cases, we can set $\lambda$ to any high number. We used $\lambda=100$ for all the noiseless experiments. We set $\lambda$ to ${10, 1, 0.5}$ for SNR values of ${60, 40, 30}$, respectively.}

\paragraph{Noiseless data}

First, we compare the performances 
for different values of purity parameters $p$ in the noiseless case. Due to the randomness of the data generation process, the reported results are the average over 10 trials. We evaluate the ERR metric for three cases of $r=m=\{3,4,5\}$ vs 7 different purity values $p \in [\frac{1}{r-1}+0.01, 1]$. 
For the data generation, we set $n_1=30 \times r$ (30 samples on each facet) and $n_2=10$ (10 samples within the simplex) for a total of $n=30 \times r + 10$ samples. The average ERR and running times over 10 trails are reported in Figure ~\ref{noiseless}. 
\begin{figure*}[!htbp]
	\begin{minipage}[b]{0.5\linewidth}
		\centering
		\centerline{\includegraphics[width=8cm]{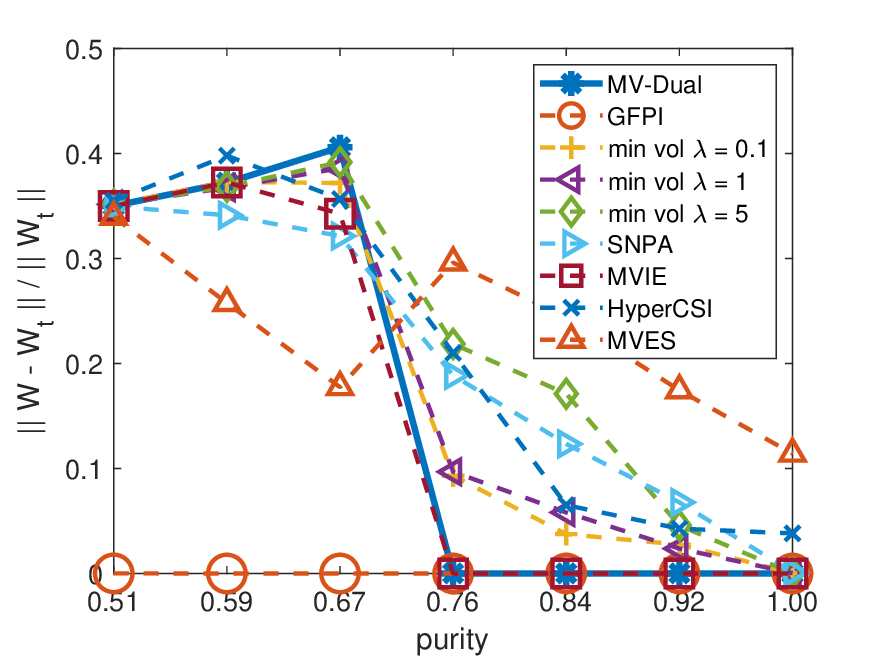}}
		\centerline{(a) ERR for $r=m=3$}\medskip
	\end{minipage}
	\hfill
	\begin{minipage}[b]{0.5\linewidth}
		\centering
		\centerline{\includegraphics[width=8cm]{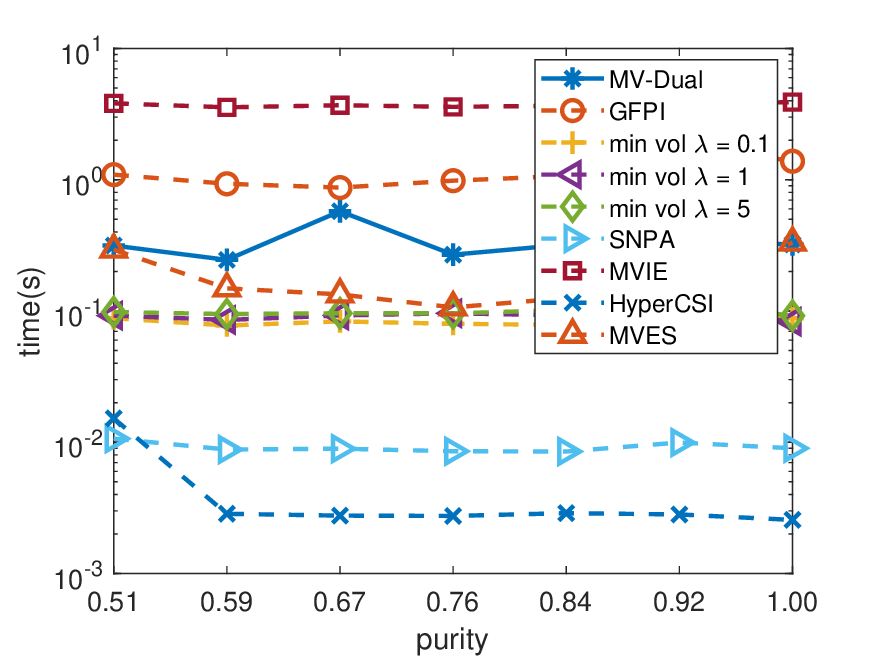}}
		\centerline{(b) Time(s) for $r=m=3$}\medskip
	\end{minipage}
	\hfill
	\begin{minipage}[b]{0.5\linewidth}
		\centering
		\centerline{\includegraphics[width=8cm]{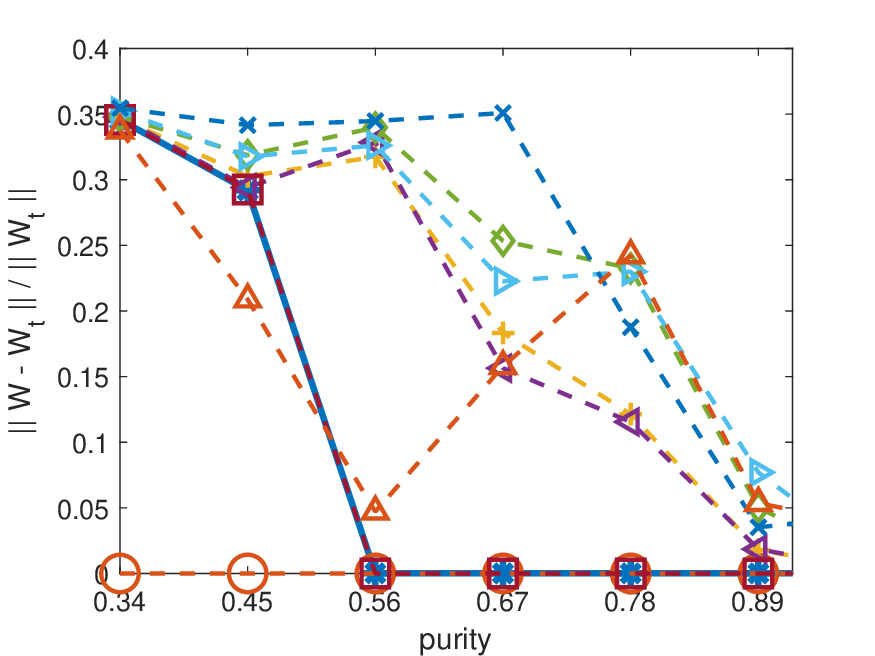}}
		\centerline{(c) ERR for $r=m=4$}\medskip
	\end{minipage}
	\hfill
	\begin{minipage}[b]{0.5\linewidth}
		\centering
		\centerline{\includegraphics[width=8cm]{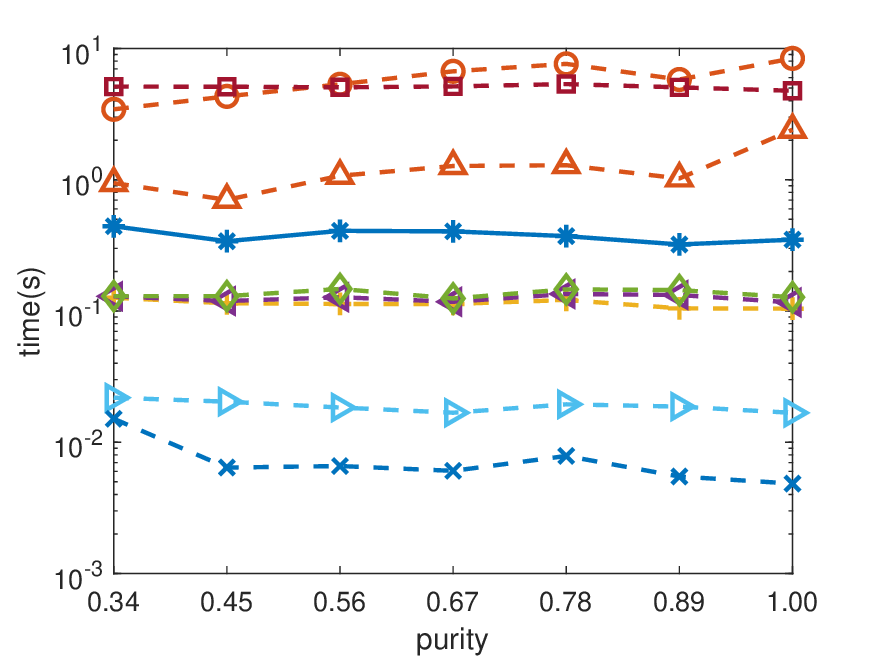}}
		\centerline{(d) Time(s) for $r=m=4$}\medskip
	\end{minipage}
	\begin{minipage}[b]{0.5\linewidth}
		\centering
		\centerline{\includegraphics[width=8cm]{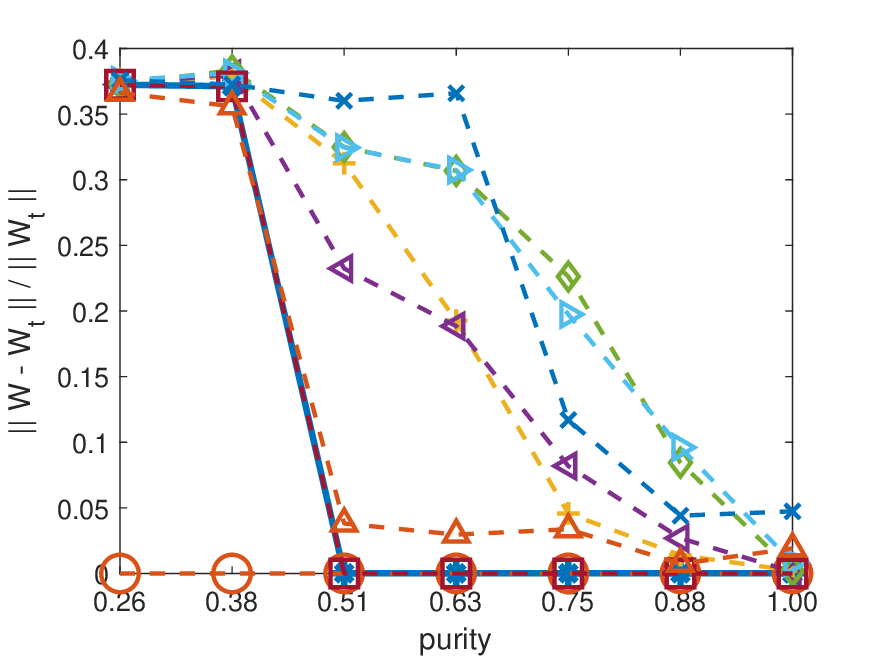}}
		\centerline{(e) ERR for $r=m=5$}\medskip
	\end{minipage}
	\hfill
	\begin{minipage}[b]{0.5\linewidth}
		\centering
		\centerline{\includegraphics[width=8cm]{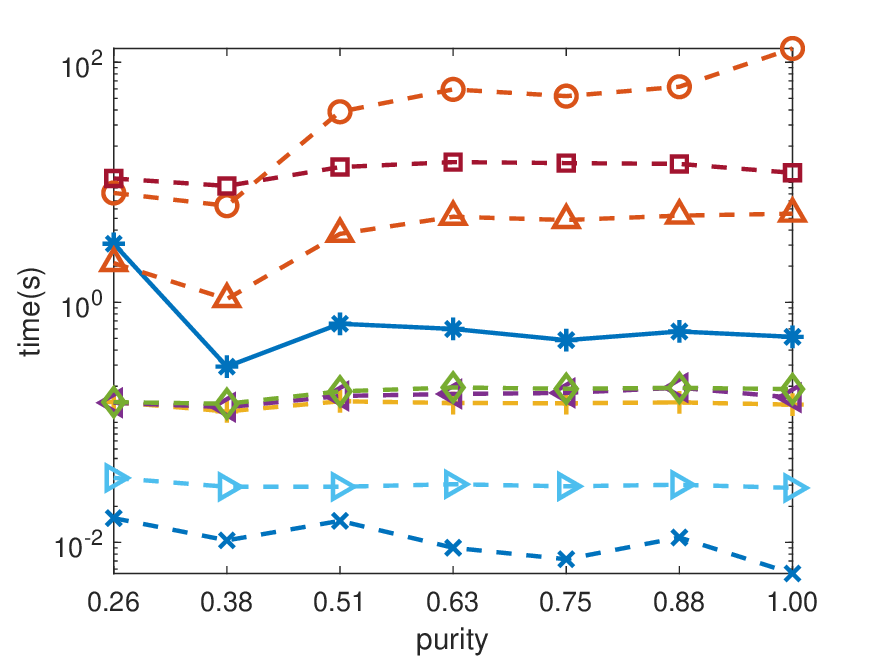}}
		\centerline{(f) Time(s) for $r=m=5$}\medskip
	\end{minipage}
	\caption{Average ERR metric and running time (in seconds)vs purity over 10 trials for noiseless data and different values of $r$ and $m$. \label{noiseless} 
 } 
\end{figure*} 
We observe that:
\begin{itemize}
	\item MV-Dual performs as well as MVIE and has significantly lower computational time.
 
	\item GFPI achieves perfect recovery of ground-truth factors for all purity levels in all cases. However, the run time of GFPI is significantly larger as it relies on solving mixed integer programs. 
 
	\item Min-vol performs better than SNPA for purities less than one, but does not recover the ground-truth factors even when the SSC condition is satisfied.
 
\end{itemize}

For low values of the purity, only GFPI performs perfectly. The reason is that the data does not satisfy the SSC, and there exists smaller volume solutions (but with less points on their facets) that contain the data points. {This is illustrated for a simple example for $r=3$ in Fig~\ref{fig:mvdual-gfpi-low-purity}, where the facet-based criterion used in GFPI finds the correct endmembers, whereas the volume-based MV-Dual selects the enclosing simplex with smaller volume.}
\begin{figure*}[!htbp]
		\centering
		\centerline{\includegraphics[width=8cm]{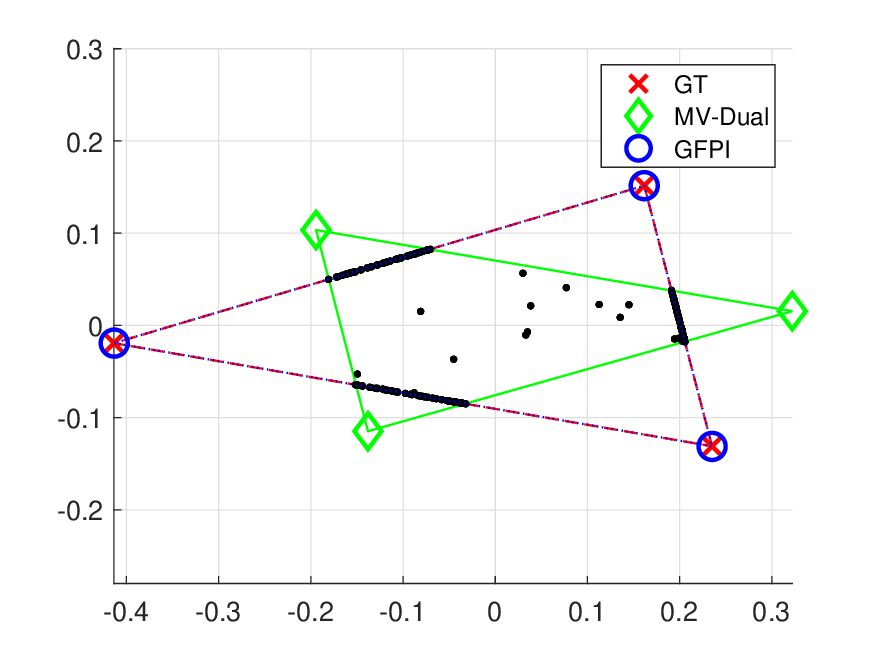}}
	\caption{MV-Dual vs GFPI in the case of low-purity. GT stands for ground truth.}
		\label{fig:mvdual-gfpi-low-purity}
\end{figure*}

\paragraph{Noisy data} 
We consider three SNRs, in \{60, 40, 30\}, for $m=r=\{3,4\}$ and generate synthetic ground-truth factors $W_t$ and $H_t$ identical to the previous noiseless experiment (with $n_1=30 \times r$ and $n_2=10$). The average ERR metric and running time over 10 trials are reported in Figure~\ref{noisy_ERR} and {the average run times  are summarized in Tables~\ref{tab:timenoisy-r-3} ($r=3$) 
and~\ref{tab:timenoisy-r-4} ($r=4$).} 
\begin{figure*}[!htbp]
	\begin{minipage}[b]{0.5\linewidth}
		\centering
		\centerline{\includegraphics[width=8cm]{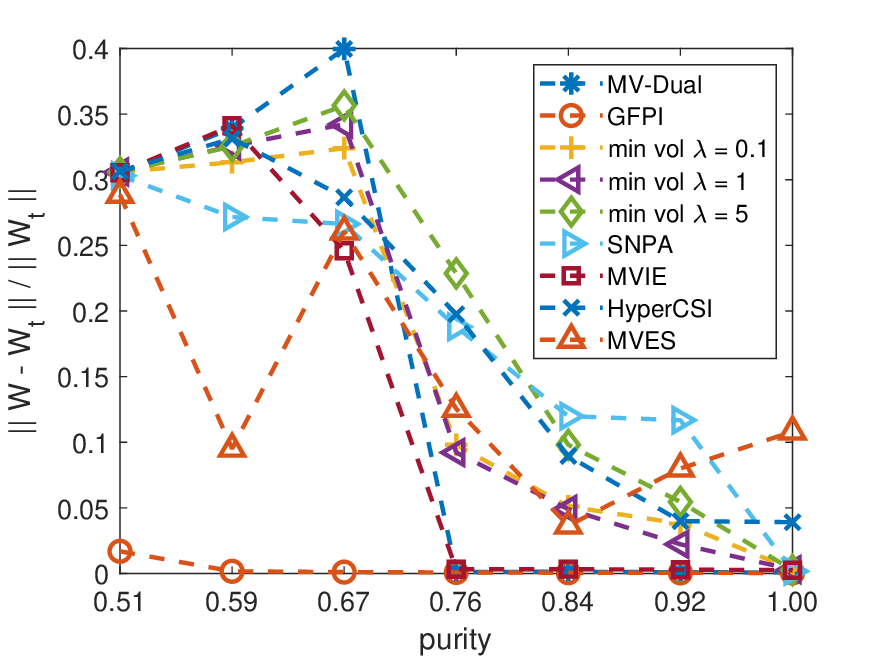}}
		\centerline{(a) $r=m=3$, SNR = 60}\medskip
	\end{minipage}
	\hfill
	\begin{minipage}[b]{0.5\linewidth}
		\centering
		\centerline{\includegraphics[width=8cm]{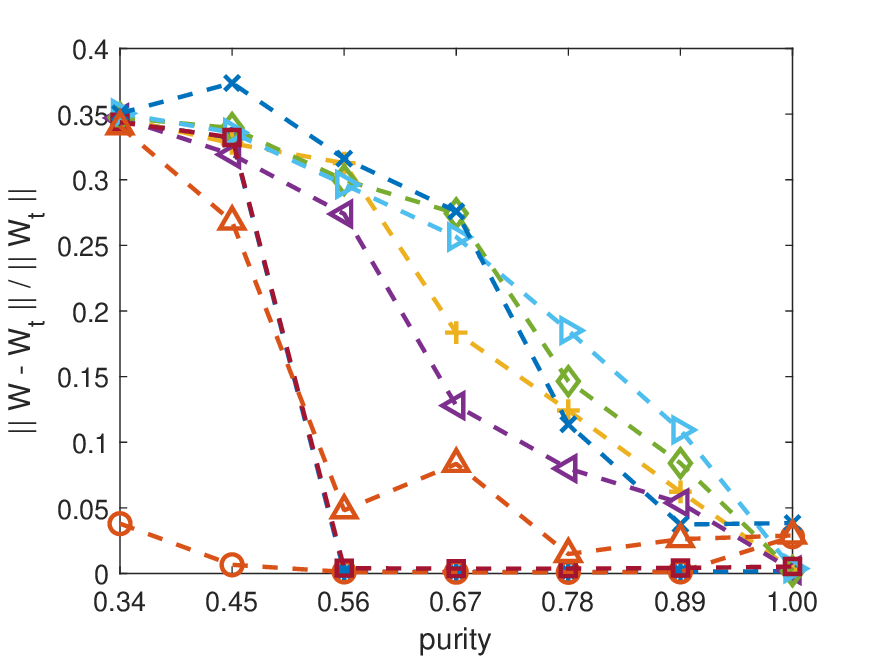}}
		\centerline{(b) $r=m=4$, SNR = 60}\medskip
	\end{minipage}
	\hfill
	\begin{minipage}[b]{0.5\linewidth}
		\centering
		\centerline{\includegraphics[width=8cm]{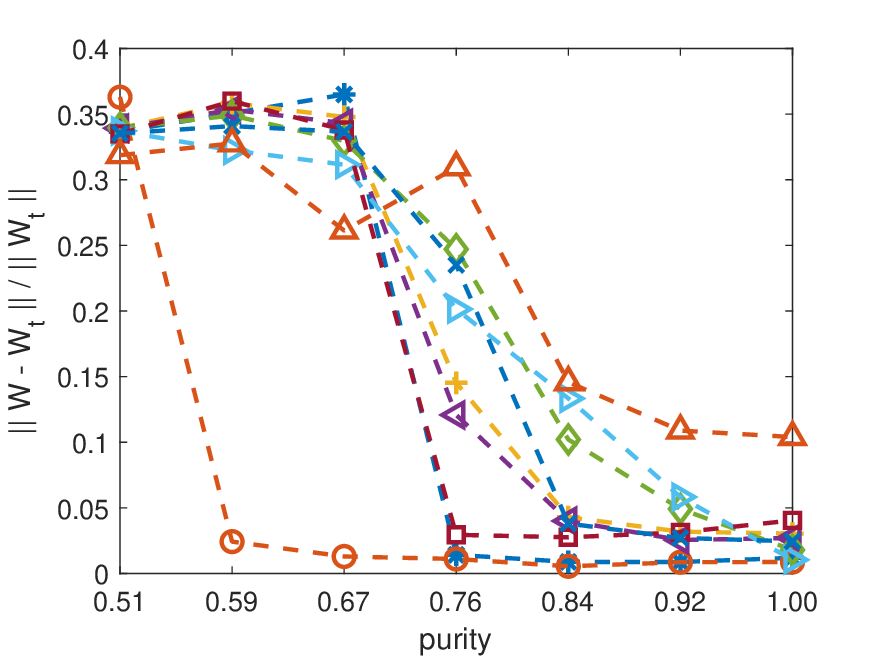}}
		\centerline{(c) $r=m=3$, SNR = 40}\medskip
	\end{minipage}
	\hfill
	\begin{minipage}[b]{0.5\linewidth}
		\centering
		\centerline{\includegraphics[width=8cm]{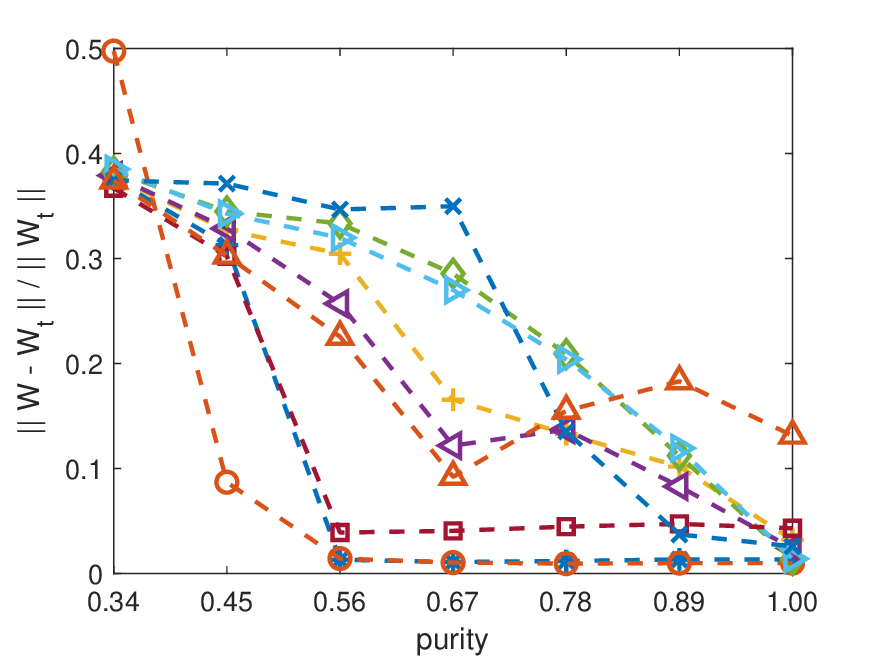}}
		\centerline{(d) $r=m=4$, SNR = 40}\medskip
	\end{minipage}
	\begin{minipage}[b]{0.5\linewidth}
		\centering
		\centerline{\includegraphics[width=8cm]{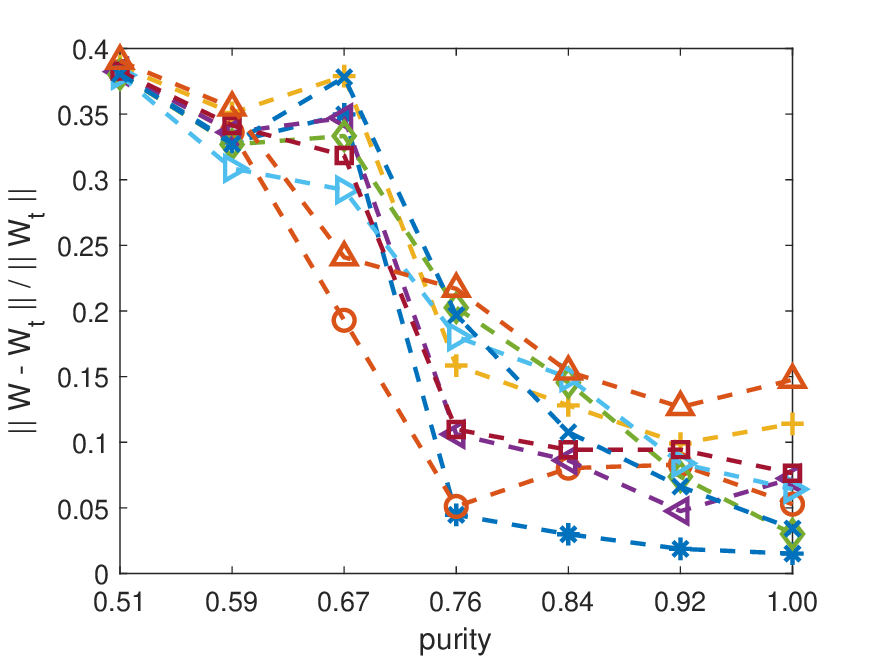}}
		\centerline{(e) $r=m=3$, SNR = 30}\medskip
	\end{minipage}
	\hfill
	\begin{minipage}[b]{0.5\linewidth}
		\centering
		\centerline{\includegraphics[width=8cm]{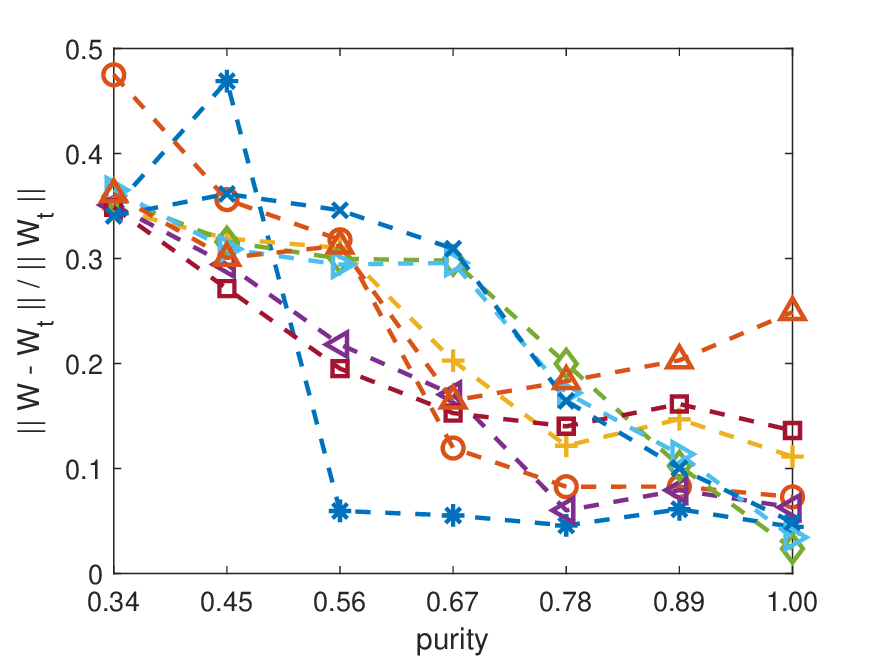}}
		\centerline{(f) $r=m=4$, SNR = 30}\medskip
	\end{minipage}
	\caption{Average ERR metric vs purity over 10 trials for noisy data and different values of $r$, $m$ and SNR levels. \label{noisy_ERR}} 
\end{figure*}

\begin{center}
	\begin{table*}[!htbp]
		\begin{center}
            \small 
			\tabcolsep=0.07cm
			\begin{tabular}{c||c|c|c|c|c|c|c|c|c}
				& MVDual & GFPI  &  min vol & min vol & min vol & SNPA & MVIE & HyperCSI & MVES  \\
               SNR & & & $\lambda=0.1$ & $\lambda=1$ & $\lambda=5$ & & & &  \\\hline
				30 & 0.56$\pm$0.11 & 7.76$\pm$3.51 & 0.12$\pm$0.01 & 0.13$\pm$0.01 & 0.14$\pm$0.02 & 0.01$\pm$0.001 & 5.28$\pm$0.23 & 0.01$\pm$0.004 & 0.30$\pm$0.04 \\
				40 & 0.45$\pm$0.06 & 4.18$\pm$1.12 & 0.10$\pm$0.01 & 0.11$\pm$0.01 & 0.13$\pm$0.01 & 0.01$\pm$0.00 & 4.96$\pm$0.12 & 0.005$\pm$0.004 & 0.30$\pm$0.05   \\
                60 & 0.42$\pm$0.06 & 1.47$\pm$0.45 & 0.07$\pm$0.01 & 0.08$\pm$0.01 & 0.09$\pm$0.01 & 0.01$\pm$0.00 & 3.78$\pm$0.12 & 0.001$\pm$0.00 & 0.26$\pm$0.07  \\
			\end{tabular} 
   			\caption{Average run times in seconds of SSMF algorithms on noisy synthetic data for $r=3$. \label{tab:timenoisy-r-3}} 
		\end{center}
	\end{table*}
\end{center}

\begin{center}
	\begin{table*}[!htbp]
		\begin{center}
		\small  
			\tabcolsep=0.07cm
			\begin{tabular}{c||c|c|c|c|c|c|c|c|c}
				& MVDual & GFPI  &  min vol & min vol & min vol & SNPA & MVIE & HyperCSI & MVES  \\
                SNR & & & $\lambda=0.1$ & $\lambda=1$ & $\lambda=5$ & & & &  \\\hline
				30 & 1.36$\pm$0.99 & 143.24$\pm$76.91 
                & 0.14$\pm$0.01 &
                0.15$\pm$0.01 &
                0.18$\pm$0.02 &
                0.02$\pm$0.002 &
                6.46$\pm$0.29 &
                0.01$\pm$0.004 &
                0.61$\pm$0.07 \\
				40 & 0.97$\pm$0.82 & 64.52$\pm$36.69 & 0.15$\pm$ 0.01 & 0.17$\pm$ 0.03 & 0.20$\pm$ 0.04 & 0.02$\pm$ 0.003 & 7.13$\pm$ 0.40 & 0.01$\pm$0.007 & 0.75$\pm$0.08  \\
                60 & 0.56$\pm$0.05 & 22.79$\pm$8.87 & 0.16$\pm$0.01 & 0.19$\pm$0.03 & 0.21$\pm$0.04 & 0.02$\pm$0.01 & 7.58$\pm$0.37 & 0.01$\pm$0.01 & 1.22$\pm$0.25  
			\end{tabular} 
    \caption{Average run times in seconds of SSMF algorithms on noisy synthetic data for $r=4$. \label{tab:timenoisy-r-4}} 
		\end{center}
	\end{table*}
\end{center}

We observe that:
\begin{itemize}
        \item GFPI is the most effective algorithm when the noise level is low, but it is the slowest. 
        
        \item As the noise level increases, the performances of MVIE and GFPI gets worse. This indicates that MVIE and GFPI are more sensitive to noise. In fact, for high noise level and high purity, MV-Dual performs the best. 
        
        \item MV-Dual is the second best algorithm in low noise regimes, and the most effective algorithm as the noise level increases. Moreover, MV-Dual is significantly faster than both MVIE and GFPI. 
\end{itemize}

In Appendix~\ref{app:convsensMVdual}, we discuss the convergence of MV-Dual and sensitivity to the parameter $\lambda$. In a nutshell, the conclusions are as follows: 
\begin{itemize}
    \item MV-Dual requires a few updates of the translation vector $v$ to converge, on average less than 10. 

    \item MV-Dual is not too sensitive to the choice of $\lambda$. 
\end{itemize}

\subsection{Unmixing hyperspectral data}

We apply SSMF algorithms for the unmixing problem on two real-world hyperspectral images: Samson and Jasper Ridge~\cite{zhu2017hyperspectral}. The goal is to identify the so-called pure pixels (a.k.a.\ endmembers) which are the columns of $W$, while the weight matrix $H$ contains the abundances of these pure pixels in the pixels of the image. 
To compare the performance, we use two metrics usually used in this literature: 
\begin{itemize}

	\item Mean Removed Spectral Angle (MRSA) between two vectors $x \in \mathbb{R}^{m}$ and $y \in \mathbb{R}^{m}$ is defined as 
	\[ 
  \text{MRSA}(x,y)=\frac{100}{\pi}\cos^{-1}\left(\frac{(x-\bar{x}e)^\top (y-\bar{y}e)}{\|x-\bar{x}e\|_2 \|y-\bar{y}e\|_2}\right),\]
	where $\bar{x}=\frac{1}{n}\sum_{i=1}^n x(i)$. 
 We will report the average MRSA between the columns of $W$ (permuted to minimize that quantity) and $W_t$. 
 
	\item Relative Reconstruction Error (RE): measures how well the data matrix is reconstructed using $W$ and $H$. 
	\[ 
 \text{RE} = \frac{\|X-WH\|_F}{\|X\|_F}. 
 \]
\end{itemize}

\paragraph{Samson data set}

The Samson image has $95\times 95$ pixels, each with 156 spectral bands and contains  three endmembers ($r=3$): ``soil", ``water" and ``tree"~\cite{zhu2017hyperspectral}. 
The solution obtained by MV-Dual is illustrated in Figure~\ref{samson_W}.  
We compare the performance of MV-Dual to other SSMF algorithms in Table~\ref{tab:samson}. {We set $\lambda=0.002$ in this experiment.}
\begin{figure*}[!htbp]
	\begin{minipage}[b]{0.5\linewidth}
		\centering
		\centerline{\includegraphics[width=8cm]{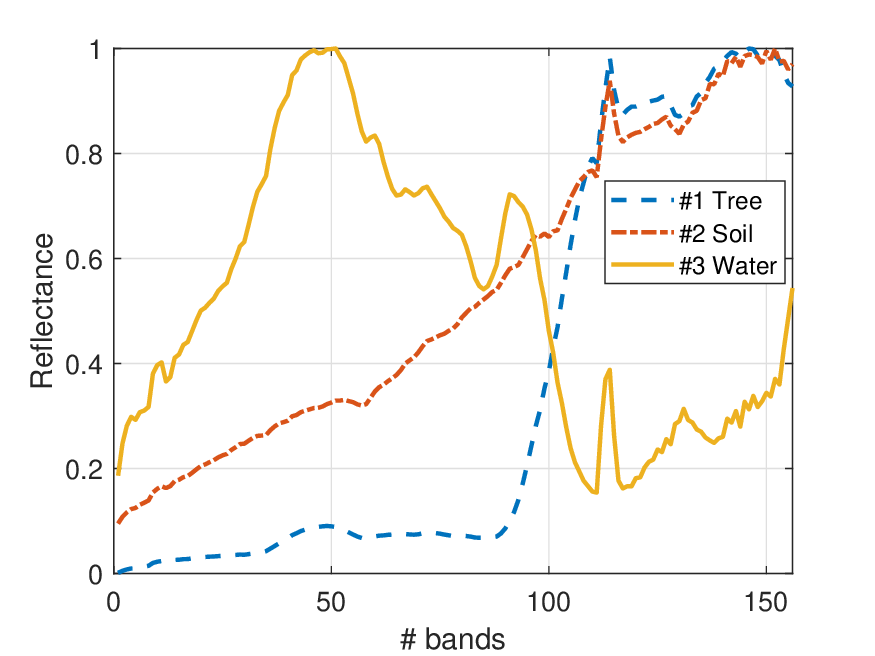}}
		\center{(a) Spectral signatures of the estimated endmembers.}\medskip
	\end{minipage}
	\hfill
	\begin{minipage}[b]{0.5\linewidth}
		\centering
		\centerline{\includegraphics[width=8cm]{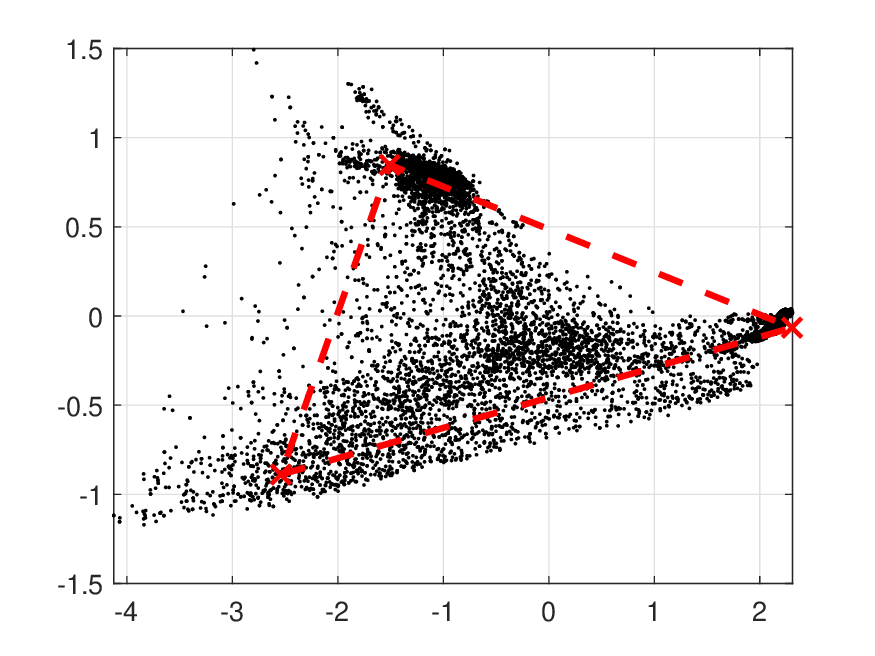}}
		\center{(b) Two-dimensional projection of the data points (dots), and the polytope computed by MV-Dual. }\medskip
	\end{minipage}
	\hfill
	\begin{minipage}[b]{1\linewidth}
		\centering
		\centerline{\includegraphics[width=15cm]{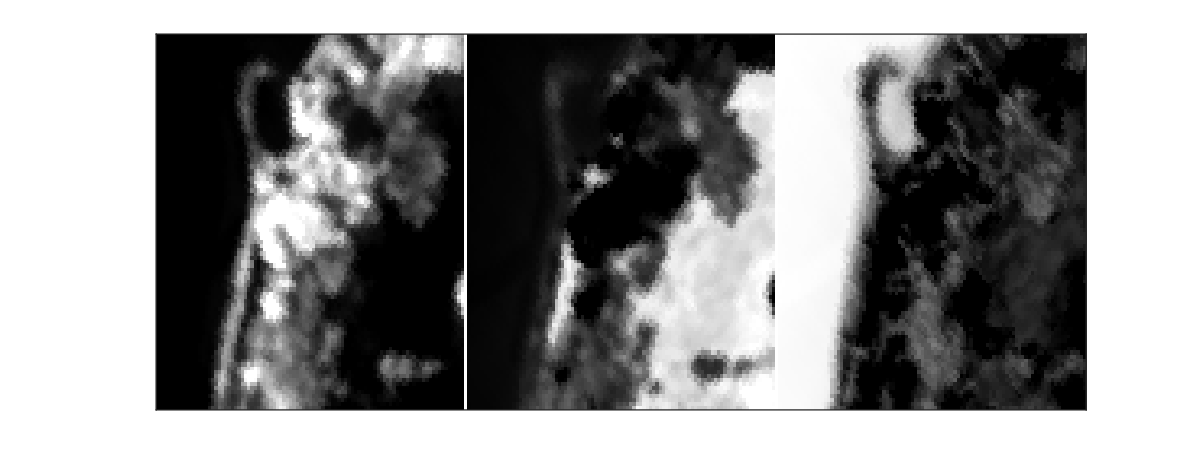}}
		\center{(c) Abundance maps estimated by MV-Dual. From left to right: soil, tree and water. 
		}\medskip
	\end{minipage}
	\caption{MV-Dual applied on the Samson hyperspectral image.}
		\label{samson_W} 
\end{figure*}
\begin{center}
	\begin{table*}[!htbp]
		\begin{center}
			\caption{Comparing the performances of MV-Dual with state-of-the-art SSMF algorithms on Samson data set. Numbers marked with * indicate that the corresponding algorithms did not converge within 100 seconds.}
			\label{tab:samson}  
			\small\addtolength{\tabcolsep}{-1pt}
			\begin{tabular}{c|ccccc}
				& SNPA  & Min-Vol & HyperCSI & GFPI & MV-Dual \\\hline
				MRSA & 2.78 & 2.58 & 12.91 & 2.97 & 2.50 \\
				$\frac{||X-WH||_F}{||X||_F}$ & 4.00\% & 2.69\% & 5.35\% & 4.02\% & 5.81\% \\
				Time (s) & 0.37 & 1.30 & 0.90 & $100^*$ & 15.78
			\end{tabular} 
		\end{center}
	\end{table*}
\end{center}
MV-Dual has the best MRSA, slightly better than Min-Vol, and has a larger computational time. This is expected as Min-Vol uses specialized first-order algorithm for the optimization, whereas MV-Dual uses the generic \textit{quadprog} method of Matlab within each iteration of the optimization procedure. 
Moreover, MV-Dual has a higher relative error: this is expected since, as opposed to Min-Vol, it does not directly minimize this quantity.

\paragraph{Jasper-ridge data set}

The Jasper Ridge data set consists of $100 \times 100$ pixels with 224 spectral bands, with four endmembers ($r=4$) in this image: ``road", ``soil", ``water" and ``tree"~\cite{zhu2017hyperspectral}. Similar to the Samson data set, we plot the extracted endmembers, projected fitted convex hull and abundance maps obtained by MV-Dual in Figure~\ref{Jasper_W}. {We set $\lambda=0.0015$ in this experiment.} The detailed numerical comparison with other algorithms is reported in Table~\ref{jaspert}.
\begin{figure*}[!htbp]
	\begin{minipage}[b]{0.5\linewidth}
		\centering
		\centerline{\includegraphics[width=8cm]{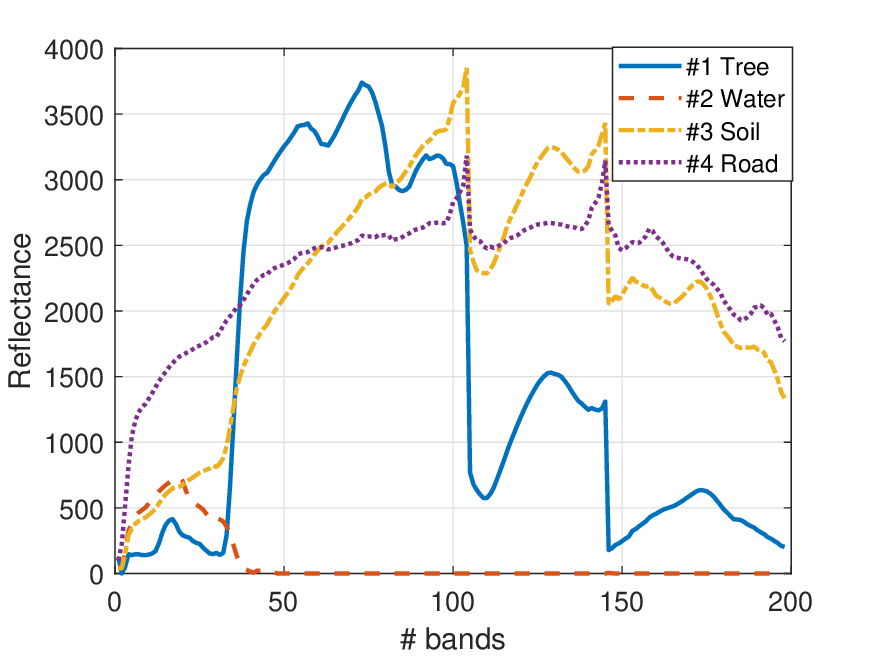}}
		\center{(a) Spectral signatures of the estimated endmembers.}\medskip
	\end{minipage}
	\hfill
	\begin{minipage}[b]{0.5\linewidth}
		\centering
		\centerline{\includegraphics[width=8cm]{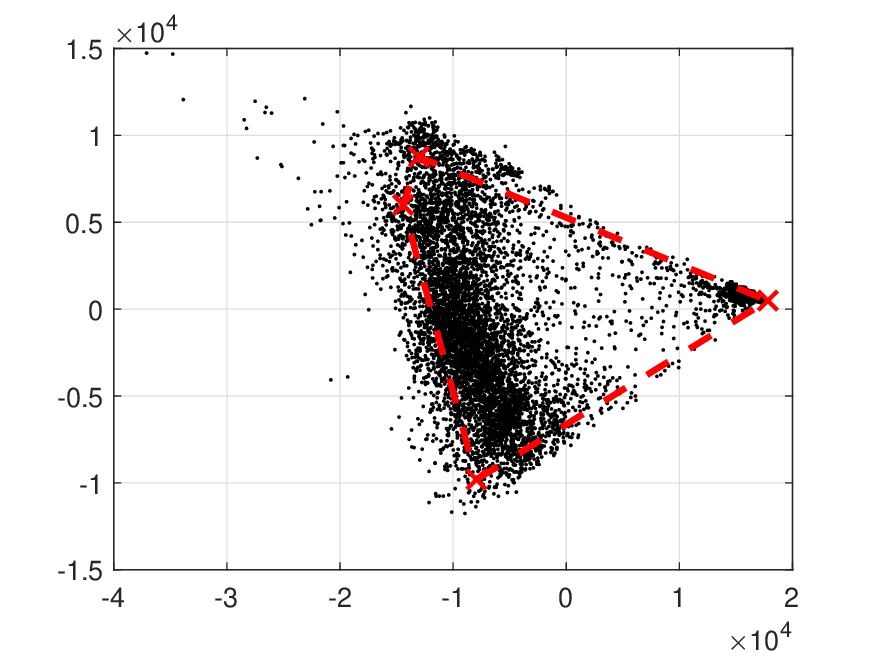}}
		\center{(b) Two-dimensional projection of the data points (dots), and the polytope computed by MV-Dual. }\medskip
	\end{minipage}
	\hfill
	\begin{minipage}[b]{1\linewidth}
		\centering
		\centerline{\includegraphics[width=15cm]{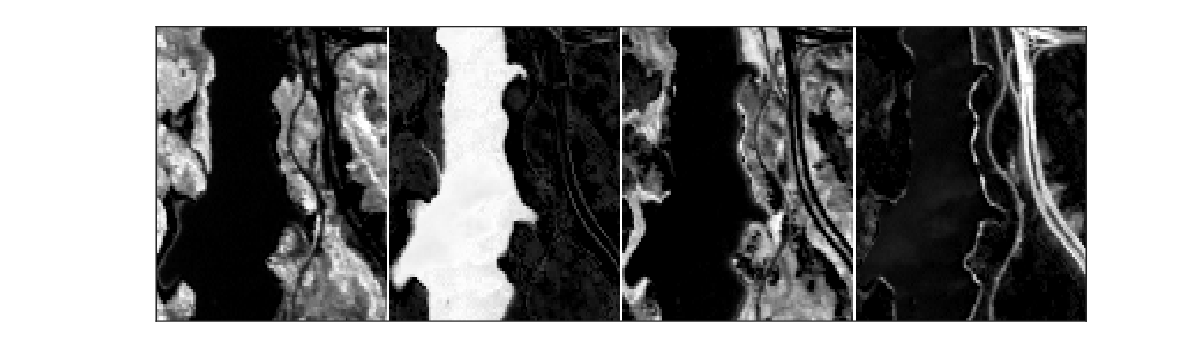}}
		\center{(c) Abundance maps estimated by MV-Dual. From left to right: road, tree, soil, water.
		}\medskip
	\end{minipage}
	\caption{MV-Dual applied on the Jasper-ridge hyperspectral image.
		\label{Jasper_W}
	} 
\end{figure*}

\begin{center}
	\begin{table*}[!htbp]
		\begin{center}
			\caption{Comparing the performances of MV-Dual with the state-of-the-art SSMF algorithms on Jasper-Ridge data set. Numbers marked with * indicate that the corresponding algorithms did not converge within 100 seconds.}
			\label{jaspert}  
			\addtolength{\tabcolsep}{-1pt}
			\begin{tabular}{c|ccccc}
				& SNPA  &  Min-Vol & HyperCSI & GFPI & MV-Dual \\\hline
				MRSA & 22.27 & 6.03 & 17.04 & 4.82 & 3.74\\
				$\frac{||X-WH||_F}{||X||_F}$ & 8.42\% & 6.09\% & 11.43\% & 6.47\% & 6.21\% \\
				Time (s) & 0.60 & 1.45 & 0.88 & $100^*$ & 43.51 
			\end{tabular} 
		\end{center}
	\end{table*}
\end{center}
The conclusions are similar as for the previous data set: MV-Dual has the best performance in terms of MRSA, here significantly smaller than Min-Vol, while the relative error is worse, but very close, to that of Min-Vol, and the computational is larger but reasonable.


\section{Conclusion}

SSMF is the problem of finding a set of points whose convex hull contains a given set of data points. To make the problem meaningful and identifiable, several approaches have been proposed, the two most popular ones being to (1)~minimize the volume of the sought convex hull, 
and (2)~identify the facets of that convex hull by leveraging the fact that they should contain as many data points as possible (leading to sparse representations). 
In this paper, we have proposed a new approach to tackle SSMF by maximizing the volume of the polar of that convex hull. We showed that this approach also leads to identifiability under the same assumption as the minimum-volume approaches; namely, the sufficiently scattered condition (SSC). 
However, the two models are not equivalent, and our proposed maximum-volume approach is able to obtain more consistent solutions on synthetic data experiments, especially in high noise regimes, while having a low computational cost. We also showed that it provides competitive results to unmix real-world hyperspecral images. 

Further work include 
\begin{itemize}
    \item The implementation of dedicated and faster algorithms, with convergence guarantees, to solve our min-max formulation~\eqref{eq:minmax_volume}. 

\item A strategy to tune $\lambda$ automatically. In the paper, we used a fixed value of $\lambda$, but it would be possible to tune it, e.g., based on the relative error of the current solution. 

    \item The design of more robust models, e.g., replacing the $\ell_2$-norm based SVD preprocessing and the minimization of the Frobenius norm of $\Delta$ in~\eqref{eq:thirdformupolar} by more robust norms, e.g., 
 the component-wise $\ell_1$ norm.  
 

\item Adapt the theory and model in the rank-deficient case, that is, when Assumption~\ref{ass:dimaffhull} is not satisfied: $\conv(W)$ is not a simplex but a polytope in dimension $d$ with more than $d+1$ vertices.

\end{itemize}

\small

\bibliographystyle{spmpsci} 
\bibliography{maxVol}

\normalsize 

\section{Appendix} 

\subsection{Proofs of Lemma~\ref{lem:polar_after_translation} 
and Lemma~\ref{lem:volume_after_multiplication}} \label{app:proofs}

\begin{proof}[Proof of Lemma \ref{lem:polar_after_translation}]
For any column $\theta_i$ of $\Theta$, call $\mathcal S_i$ the corresponding face of $\mathcal S_i$, i.e. $\mathcal S_i :=\{x  \in  \mathcal S: \theta_i^\top x = 1 \}$, whose affine span is the affine subspace $\theta_i/\|\theta_i\|^2 + \theta_i^\perp$ . When we translate by $w$, the column $\theta^{(w)}_i$ of $\Theta_w$ corresponding to the face $\mathcal S_i-w$ will satisfy $(\theta^{(w)}_i)^\top (x-w) = 1$ for every $x\in \theta_i/\|\theta_i\|^2 + \theta_i^\perp$, and in particular
\[
\theta^{(w)}_i || \,\theta_i, \quad
(\theta^{(w)}_i)^\top(\theta_i/\|\theta_i\|^2-w) = 1 \implies 
\theta^{(w)}_i = \frac 1{1-\theta_i^\top w} \theta_i\implies 
\Theta_w = \Theta \diag(e-\Theta^\top w)^{-1}.
\]
Notice that if $\Theta z = 0$ and $e^\top z = 1$, then $0 = \Theta_w \diag(e-\Theta^\top w)z$ and $e^\top\diag(e-\Theta^\top w)z = e^\top z - z^\top\Theta^\top w = 1$, so $z_w = \diag(e-\Theta^\top w)z$ and  $\Theta_w\diag(z_w) = \Theta \diag(z)$.

Let now $\Theta_v$ be the polar of $A_v: =A-ve^\top$, where   $A\in \mathbb R^{(r-1)\times r}$,  $Ae=0$ and $At = v\in \conv(A)^{\circ}$ with $t\in\Delta^r$. The face of $\conv(A_v)$ associated to $\theta_i^{(v)}$ is generated by  all the columns of $A_v$ except for the $i$-th one. As a consequence, $A_v^\top\Theta_v$ has all entries equal to one except for the diagonal, and in particular it is symmetric. As a consequence, $A_v^\top\Theta_vt = \Theta_v^\top A_v t = 0$ but since $A_v^\top$ is column full rank, we find that $\Theta_v t=0$, so $z_v\equiv t$ and by the previous result $\Theta_v\diag(z_v) = \Theta_v\diag(t) = \Theta_{At}\diag(t)$ does not depend on $t$.   

\end{proof}

\begin{proof}[Proof of Lemma \ref{lem:volume_after_multiplication}]
Using the definition of volume and the multiplicativity of the determinant,
\begin{align*}
    \vol(\Theta N) &= \frac{1}{(n-1)!} \left| \det\begin{pmatrix}
        \Theta N\\ e^T
    \end{pmatrix} \right| =
    \frac{1}{(n-1)!} \left| \det\begin{pmatrix}
       \Theta\\ e^TN^{-1}
    \end{pmatrix} \right| |\det(N)|\\
    &= \frac{1}{(n-1)!} \left| \det\begin{pmatrix}
        \Theta\\ e^T
    \end{pmatrix} \right|
    \left| \det\left( I + \frac 1{e^Tw}we^T(N^{-1}-I) \right) \right|
    |\det(N)|.
\end{align*}
The eigenvalues of $ I + \frac 1{e^Tw}we^T(N^{-1}-I) $ are all equal to $1$, except, possibly, for the eigenvalue $1 +\frac 1{e^Tw}e^T(N^{-1}-I)w = \frac {e^TN^{-1}w}{e^Tw}$, thus completing the proof.

\end{proof}

Here we show that given a bounded convex polytope $\mathcal S\subseteq \mathbb R^{r-1}$ with at least $r$ vertices, the vertices of a maximum volume simplex contained in it coincide with $r$ of the vertices of $\mathcal S$. This is useful in the proof of Theorem \ref{th:minmaxformu}.

\begin{lemma} \label{ass:lemsolvert}
    Given a bounded convex polytope $\mathcal S\subseteq \mathbb R^{r-1}$ with at least $r$ vertices, the problem
    \[
   \max_{\Theta \in \mathbb{R}^{r-1 \times r}} 
\vol\big( \conv(\Theta) \big) 
\quad \text{ such that } \quad \conv(\Theta)\subseteq\conv(\mathcal S)
    \]
is solved by a matrix $\Theta$ whose columns coincide with $r$  vertices of $\mathcal S$.    
\end{lemma}
\begin{proof}
    By definition, $\vol(\conv(\Theta)) =  \frac{1}{(r-1)!} \left| \det\begin{pmatrix}
        \Theta\\ e^T
    \end{pmatrix} \right|$. If we single out a column $\theta$ of $\Theta$ and fix all other entries, we can see that $\det\begin{pmatrix}
        \Theta\\ e^T
    \end{pmatrix}  = c + q^\top \theta$ for a scalar $c\in \mathbb R$ and a vector $q\in \mathbb R^{r-1}$. The above maximization problem is equivalent to maximize $\vol(\conv(\Theta))^2$, that is proportional to $(c + q^\top \theta)^2 = \theta^\top qq^\top \theta +2cq^\top \theta + c^2$, i.e. a convex function in $\theta\in \mathcal S$. A global maximum of a convex function on a convex bounded polytope can always be found at one of its vertices. As a consequence, given any feasible $\Theta$ we can find a new feasible $\tilde \Theta$ with greater or equal volume and with vertices corresponding to a subset of the vertices of $\mathcal S$ by optimizing sequentially over the columns of $\Theta$. Since $\mathcal S$ has a finite number of vertices, then one of the maximum volume simplices contained inside $\mathcal S$ must coincide with a simplex formed by $r$ of its vertices.      
\end{proof}

\subsection{Convergence and sensitivity of MV-Dual} \label{app:convsensMVdual}

\paragraph{Convergence analysis}

Let us provide some insights on the convergence of MV-Dual. We choose $m=r=3$, and $p=0.76$ which is at the phase transition and hence leads to more difficult instances  (see Figure~\ref{noisy_ERR} in the paper). 
We explore two scenarios to generate the samples: 
\begin{enumerate}
    \item Balanced case: samples are evenly distributed across the three facets, as in the paper. We let $n_1=90$ and $n_2=10$. 

    \item Imbalanced case: samples are unevenly distributed among facets. In this case, we set $n=640$, with the number of samples on the three facets being $\{500, 100, 30\}$, and 10 samples chosen within the simplex.  
\end{enumerate}

At iteration $k$, let $v_{k}$ indicate the obtained translation vector $v$ and $W_k$ be the estimated endmembers. 
Figure~\ref{fig:convergence} shows the evolution of $\frac{\|v_{k}-v_{k-1}\|_2}{\|v_{k-1}\|_2}$ and $\frac{\|W_{k}-W_{t}\|_F}{\|W_{t}\|_F}$ where $W_t$ is the ground truth, for different iterations.  
MV-Dual converges fast for all noise levels, as less than 5 iterations are needed for $W_k$  to converge in MV-Dual. The explanation is that the solution $v_k$ does not need to attain the minimum for $W_k$ to correctly identify $W_t$; see Section~\ref{sec:identif}.  
\begin{figure*}[!htbp]
	\begin{minipage}[b]{0.45\linewidth}
		\centering
		\centerline{\includegraphics[width=8cm]{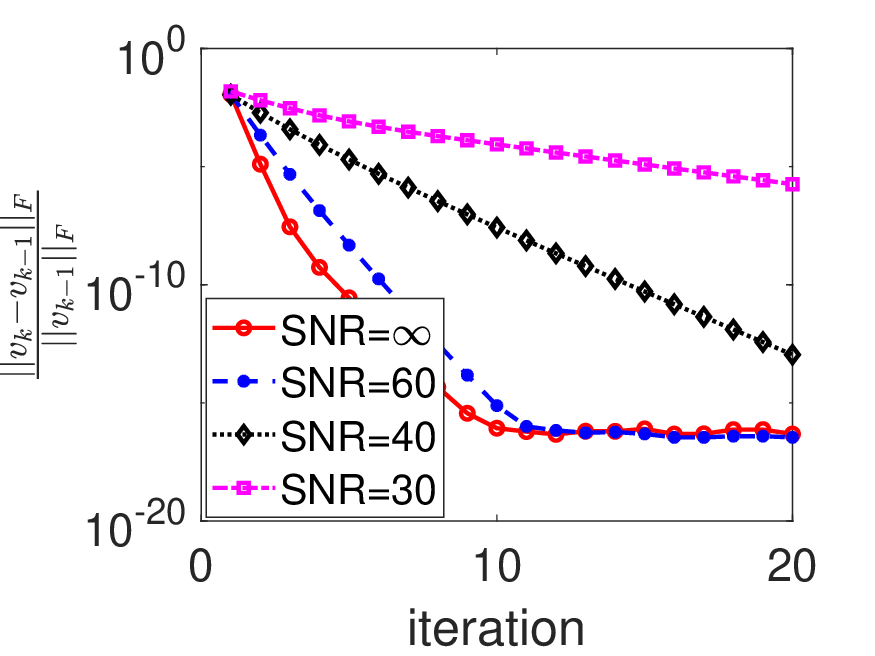}}
	\end{minipage}
	\hfill
	\begin{minipage}[b]{0.45\linewidth}
		\centering
		\centerline{\includegraphics[width=8cm]{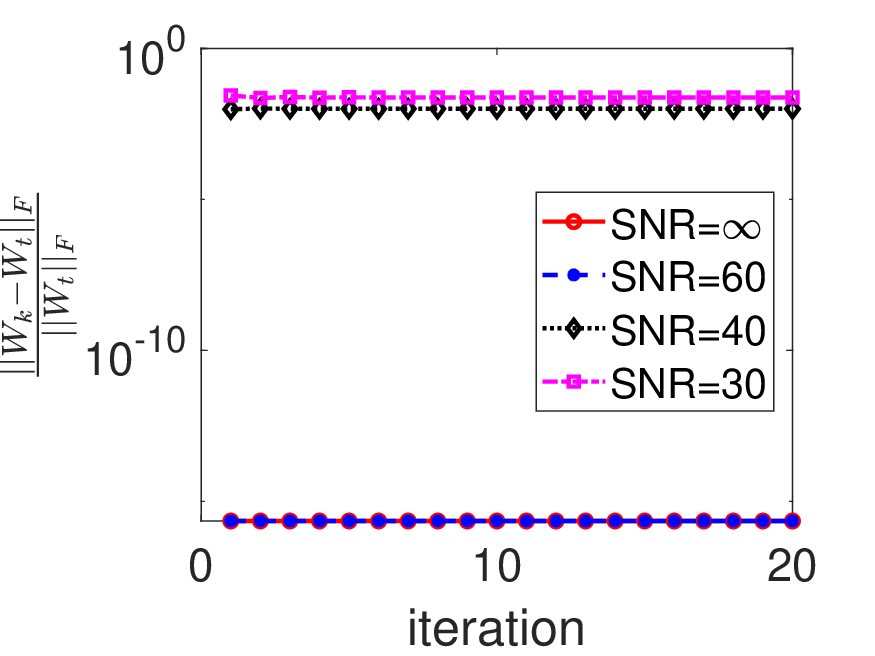}}
	\end{minipage}
        \center{(a) Balanced data}\\
	\begin{minipage}[b]{0.45\linewidth}
		\centering
		\centerline{\includegraphics[width=8cm]{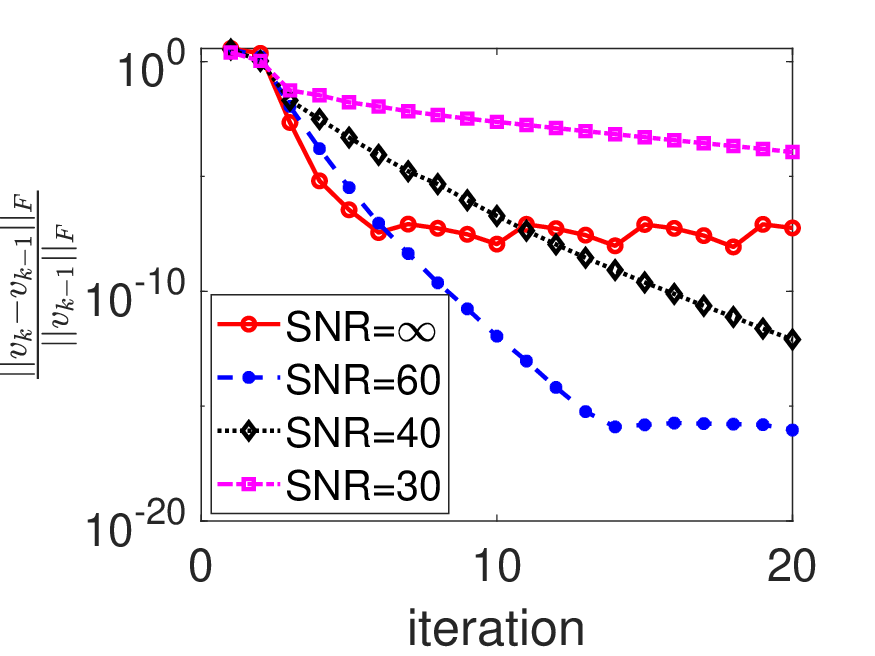}}
	\end{minipage}
 	\hfill
	\begin{minipage}[b]{0.45\linewidth}
		\centering
		\centerline{\includegraphics[width=8cm]{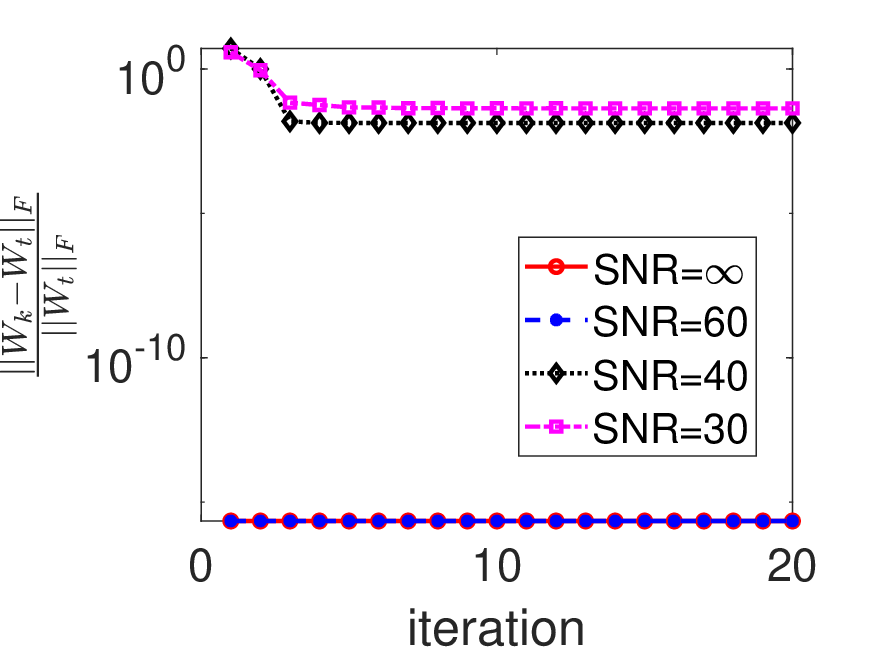}}
	\end{minipage} 
        \center{(b) Imbalanced data}
	\caption{Convergence analysis of MV-dual ($m=r=3$).}
		\label{fig:convergence}
\end{figure*}

\paragraph{Sensitivity to $\lambda$}

 Figure~\ref{fig:lambda}
 displays the average of ERR metric over 10 trials for various values of $\lambda$,  
 for $m=r=3$ and SNR=30. We observe that the performance of MV-Dual is stable w.r.t.\ the choice of $\lambda \in [0.4, 100]$, and the only noticeable sensitivity is in the 'transition' phase, with purity below 0.76, because there are two different simplices with small volumes containing the data points. 
\begin{figure*}[!htbp]
		\centering
		\centerline{\includegraphics[width=8cm]{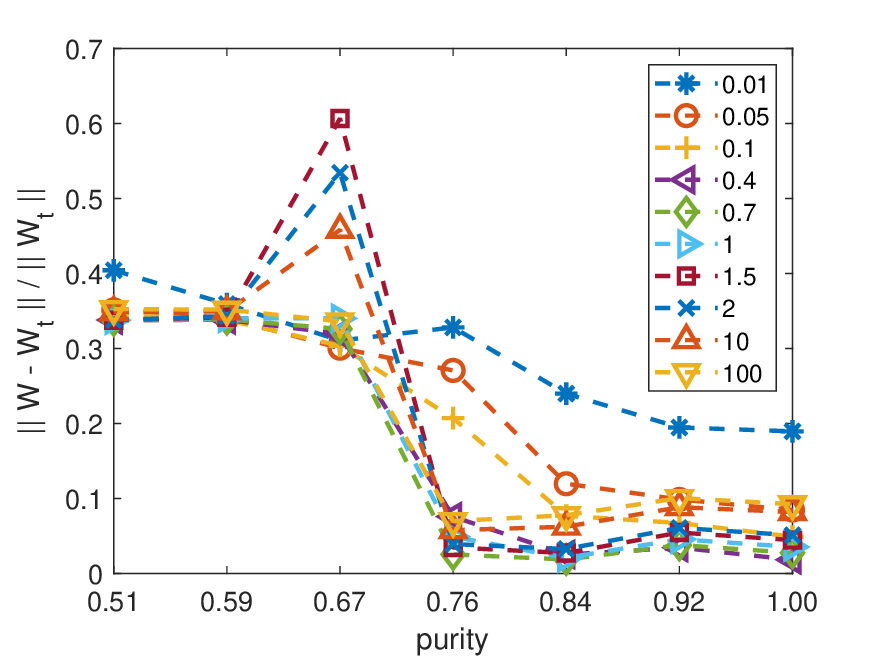}}
	\caption{Performance of MV-Dual with respect to various values of $\lambda$.}
		\label{fig:lambda}
\end{figure*}


\paragraph{MV-Dual vs GFPI}

We have seen that volume-based approaches, such as MV-Dual, outperform sample-counting-based ones, such as GFPI, in higher noise regimes, and that MV-Dual has a more stable performance. This phenomenon is illustrated on another example on Figure~\ref{fig:mvdual-gfpi}, where $m=r=3$, $\text{SNR}=30$ and there are 50 samples on each facet, with an additional 50 samples spread within the simplex. 
The worse performance of  GFPI is noticeable as the noise has perturbed the orientation of the facets, leading to a worse estimation of $W_t$.  
\begin{figure*}[!htbp]
		\centering
		\centerline{\includegraphics[width=8cm]{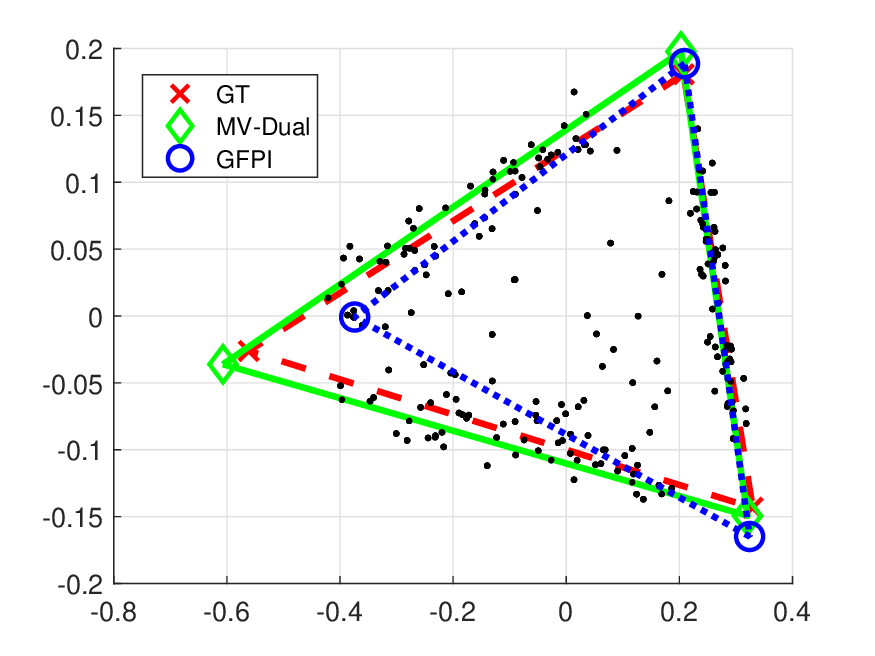}}
	\caption{MV-Dual vs GFPI. GT stands for ground truth.}
		\label{fig:mvdual-gfpi}
\end{figure*}

\end{document}